\documentclass[10pt]{amsart} \usepackage{graphicx}

\usepackage[OT1]{fontenc} \usepackage{amsmath, graphicx, mathrsfs,
amsfonts, amssymb, enumerate} \usepackage[numbers]{natbib}
\usepackage[colorlinks,citecolor=blue,urlcolor=blue]{hyperref}
\numberwithin{equation}{section}

\numberwithin{equation}{section} \theoremstyle{definition}

\theoremstyle{definition} \newtheorem{remark}[equation]{Remark}
\theoremstyle{lemma} \newtheorem{lemma}[equation]{Lemma}
\theoremstyle{theorem} \newtheorem{theorem}[equation]{Theorem}
\theoremstyle{proposition}

\theoremstyle{corollary} \newtheorem{corollary}[equation]{Corollary}
 \theoremstyle{corollary}
\theoremstyle{definition} \newtheorem{example}[equation]{Example}
\theoremstyle{example}

\newcommand{\mcG}{\mathcal{G}}
\newcommand{\mcH}{\mathcal{H}}
\newcommand{\mcC}{\mathcal{C}}
\newcommand{\mcD}{\mathcal{D}}
\newcommand{\mcI}{\mathcal{I}}
\newcommand{\mcJ}{\mathcal{J}}

\newcommand{\CC}{\mathbf{C}}
\newcommand{\ZZ}{\mathbf{Z}}
\newcommand{\NN}{\mathbf{N}}

\newcommand{\RR}{\mathbf{R}}
\newcommand{\EE}{\mathbf{E}}
\newcommand{\E}{\mathbf{E}}
\newcommand{\ip}[1]{\langle #1 \rangle}
\newcommand{\PP}{\mathbf{P}}

\DeclareMathOperator*{\sat}{\text{sat}} \DeclareMathOperator*{\ad}{ad}
 \DeclareMathOperator*{\supp}{supp}

\title[Positive Densities]{A practical criterion for positivity of transition densities} \author{David
P. Herzog} \author{Jonathan C. Mattingly}

\begin{document}

\maketitle

\begin{abstract}
We establish a simple criterion for locating points where the transition density of a degenerate diffusion is strictly positive.  Throughout, we assume that the diffusion satisfies a stochastic differential equation (SDE) on $\RR^d$ with additive noise and polynomial drift.  In this setting, we will see that it is often that case that local information of the flow, e.g. the Lie algebra generated by the vector fields defining the SDE  at a point $x\in \RR^d$, determines where the transition density is strictly positive.  This is surprising in that positivity is a more global property of the diffusion.  This work primarily builds on and combines the ideas of  Ben Arous and L\'{e}andre \cite{BAL} and Jurdjevic and Kupka \cite{JK85}.        
\end{abstract}


\section{Introduction}

The goal of this paper is to develop an easily applicable framework
for locating points where the probability density of a degenerate
diffusion is strictly positive.  We will focus on the setting where
the diffusion satisfies a stochastic differential equation (SDE) on
$\RR^d$ where each component of the drift is a polynomial in the
standard Euclidean coordinates and the noise is additive.  Our methods
reduce finding points of positivity to computing a certain collection
of constant vector fields generated by taking
iterated commutators of the vector fields defining the SDE.  This is convenient since a similar computation is typically used to
show that the diffusion has a smooth probability density function
$p_t(x,y)$ with respect to Lebesgue measure $dy$.   While the existence of a smooth density is decided locally, we show that in some settings the bracket computation also determines the more global property of where the density is strictly positive.  Additionally, uncovering sufficiently large regions of positivity is useful for proving unique ergodicity.

%

While methods already exist for proving positivity of transition
densities, most require knowledge of attainable sets via
controls. Here we have structured our assumptions to require as
little global control information as possible. In particular, our results
prove smoothness of the densities, the needed control statements, and
positivity, all with one set of primarily local assumptions.

Although our general framework is limited to SDEs with polynomial drift and additive noise, working within such boundaries is reasonable in many applications.  In particular, to illustrate the utility of our results, we will apply them to a collection of examples, each with quite different structure.  Moreover, for the equations considered, either new results will be obtained or existing results will be improved upon.

The ideas used in this note build on a number existing works. Beyond
the now classical theory of H{\"o}rmander \cite{Hor67} on hypoelliptic
operators in the ``sum of squares'' form, we use the associated
probabilistic techniques of Malliavin calculus \cite{Nua98}. We also
use a number of ideas from geometric control theory \cite{Jur97}. Moreover, we modify the idea that odd powered polynomial vector fields are ``good'' (due to their time reversal properties) and even
powered polynomial vector fields are ``bad" \cite{JK85}. Similar ideas were critical in the work of Romito \cite{Rom04}. We also integrate into
our results the powerful ideas of Ben Arous and L{\'e}andre \cite{BAL}
for proving positivity of densities of random variables over a Wiener
space. Our hope is that by bringing these ideas together and
adapting them to our specific context, we will provide a useful tool for many applied equations.

The layout of this paper is as follows.  In Section
\ref{sec:mainresults}, we introduce notation and terminology and state
the main general results of the paper.  In Section \ref{sec:examples},
we apply our results to specific examples.  Section \ref{sec:pres}
contains heuristic discussions of why the main results hold and are
natural. We also include an ``non-example'', that is an example where the main results fail to apply yet the corresponding density has regions of positivity (in space and time), and illustrate how to adapt the general theory in such cases.  Additionally, Section
\ref{sec:pres} contains the proof of  the main results as stated in 
Section \ref{sec:mainresults}.

\section*{Acknowledgements}

The authors would like to thank Avanti Athreya, Richard Durrett,
Tiffany Kolba, James Nolen, and Jan Wehr for helpful conversations on
the topic of this paper.  DPH would also like to thank Martin Hairer
for suggesting the paper \cite{JK85}, from which his understanding of
these ideas began and lead to the current collaboration.  We would
also like to acknowledge partial support of the NSF through grant
DMS-08-54879 and the Duke University Dean's office.

\section{Notation, Terminology and Main Results}
\label{sec:mainresults}

Throughout, we study stochastic differential equations on $\RR^d$ of
the following form
\begin{eqnarray}\label{eqn:sde1}
dx_t &=&X_0(x_t) \, dt +  \sum_{j=1}^{r}  X_{j}\, dW_t^{j}
\end{eqnarray}
where $X_0$ is a \emph{polynomial vector field}; that is,
$X_0=\sum_{j=1}^{d}X^{j}_0(x)\partial_{x_j}$ is such that each map
$x\mapsto X^{j}_0(x)$ is a polynomial in the standard Euclidean
coordinates, $X_{1}, \ldots, X_{r}$ are constant vector fields; that
is, they do not depend on the base point, and $W^1_t, W^2_t, \ldots,
W^r_{t}$ are standard independent real Wiener processes defined on a
probability space $(\Omega, \mathscr{F}, \PP)$.

To deal with the issue of finite-time explosion in \eqref{eqn:sde1},
we will need to stop the process $x_t$ prior to the time of
explosion. Thus for $n\in \NN$, let $B_n(0)$ denote the open ball of radius $n$ centered at the origin in $\RR^d$, and define 
the stopping times $\tau_n = \inf\{t>0 \,: \, x_t \notin B_n(0)\}$ and $\tau_\infty=\lim_{n\uparrow \infty} \tau_n$.  Our results will be stated for the
stopped processes $x_{t\wedge \tau_n}$, $n\in \NN$. Of course,
$x_{t\wedge \tau_n}$  coincides with $x_t$ for all times $t\leq \tau_n$.

For vector fields $V=\sum_{j=1}^{d} V^j(x) \frac{\partial}{\partial
  x_j}$ and $W=\sum_{j=1}^{d}W^{j}(x) \frac{\partial}{\partial x_j}$,
let $\text{ad}^{0}V(W)=W$,
\begin{equation*}
  \text{ad}^{1}V(W)=[V,W]:= \sum_{j=1}^{d} \bigg(\sum_{k=1}^{d}
  V^{k}(x) \frac{\partial W^{j}(x)}{\partial x_k} - W^{k}(x)
  \frac{\partial V^{j}(x)}{\partial x_k}\bigg)\frac{\partial}{\partial
    x_j}. 
\end{equation*}
Inductively, for $m\geq 2$ we let $\ad^{m}V(W)=\text{ad}V
\text{ad}^{m-1}V(W) $.  For a set of vector fields $\mcG$ on $\RR^d$,
$\text{span}(\mcG)$ denotes the $\RR$-linear span of $\mcG$ and
\begin{align*}
\text{cone}_{\geq 0 }(\mcG) = \{ \textstyle{\sum}_{i=1}^j \lambda_i
V_i \, : \, \lambda_i \geq 0, \, V_i \in \mcG\}.   
\end{align*}
We call $x\in \RR^d$ an \emph{equilibrium point} of a set of vector
fields $\mcG$ if $V(x)=0$ for some $V\in \mcG$.  If $V$ is a constant
vector field with constant value $v\in \RR^{d}$ and $W$ is a
polynomial vector field, then we may define a map from $\RR$ into
$\RR^{d}$ given by $\lambda \mapsto (W^j(\lambda v))$.  Note that
since $W$ is a polynomial vector field, $(W^j(\lambda v))$ is a vector
of polynomials in $\lambda$.  Let $n(V,W)$ be the maximal degree among
these polynomials (For purposes below, we assume that the zero
polynomial has neither even nor odd degree).  We call $n(V,W)$ the
\emph{relative degree} of $V$ and $W$.

We now introduce the set of constant vector fields $\mcC$ which will
play a fundamental role throughout the paper.  It will be defined as
the subset of constant vector fields in a larger set of vector fields
which we now introduce.  To initialize the inductive procedure let
$\mcG_0=\text{span}\{X_1, \ldots , X_r\} $ and
\begin{align*}
\mcG_{1}^{\text{o}} &= \mcG_0 \cup  \{\text{ad}^{n(V,X_0)}V(X_0) \,
: \, V \in \mcG_0 ,  n(V,X_0) \text{ odd} \},\\ 
\mcG_{1}^{\text{e}}&=\{\text{ad}^{n(V,X_0)}V(X_0) \, : \,V \in
\mcG_0, \, n(V,X_0) \text{ even} \},\\ 
\mcG_1 &= \text{span} (\mcG_{1}^{\text{o}}) + \text{cone}_{\geq 0}(
\mcG_{1}^{ \text{e}}).   
\end{align*} 
For $j\geq 1$, we define $\mcG_{j+1}^{\text{o}},\,
\mcG_{j+1}^{\text{e}}, \, \mcG_{j+1}$ inductively as  
\begin{align*} 
  \mcG_{j+1}^{\text{o}} &=\mcG_{j}^{\text{o}}\cup
  \{\text{ad}^{n(V,W)}V(W) \, : \, V \in \mcG_{j}^{\text{o}}  \text{
    constant}, \, W\in \mcG_{j}, \, n(V,W) \text{ odd} \},\\ 
  \mcG_{j+1}^{\text{e}}&= \mcG_{j}^{\text{e}}\cup
  \{\text{ad}^{n(V,W)}V(W) \, : \, V \in \mcG_{j}^{\text{o}} \text{
    constant}, \, W\in \mcG_{j}, \, n(V, W) \text{ even} \},\\ 
  \mcG_{j+1} &= \text{span}(\mcG_{j}^{\text{o}} ) +
  \text{cone}_{\geq 0}( \mcG_{j}^{\text{e}}).   
\end{align*}
Let $\mcC^\text{o}$ denote the set of constant vector fields in
$\cup_j \mcG_{j}^{\text{o}}$ and $\mcC^{\text{e}}$ denote the set
of constant vector fields in $\cup_j \mcG_{j}^{\text{e}}$.
Finally, define  
\begin{align}
\mcC = \text{span}(\mcC^\text{o}) + \text{cone}_{\geq 0}(\mcC^\text{e}).  
\end{align}

\begin{remark}
  Throughout, we will often identify a constant vector field on
  $\RR^d$ with the vector in $\RR^d$ which defines it.  For example,
  depending on the context, $\mcC^{\text{o}}$ will be used to denote
  either the set of vector fields $\mcC^{\text{o}}$ defined above or
  the set of vectors $v\in \RR^d$ such that $v=V(x)$ for some $V\in
  \mcC^{\text{o}}$.
\end{remark}

\begin{remark}
  The primary assumption we will make is that $\mcC$ is
  $d$-dimensional.  This is equivalent to assuming that $\mcC$ spans the entire tangent space at all points $x\in \RR^d$ as $\mcC$ contains only constant vector fields.  Since $\mcC$
  is contained in the Lie algebra generated by
\begin{align*}
X_1, \ldots, X_r, [X_1, X_0], \ldots, [X_r, X_0], 
\end{align*}  
it follows by H\"{o}rmander's hypoellipticity theorem \cite{Hor67}
that for every $n\geq 1$, $x\in B_n(0)$  and every Borel set $A\subset
B_n(0)$
\begin{align*}
\PP_x \{ x_{t\wedge \tau_n} \in A\} = \int_A p_t^n(x,y)\, dy 
\end{align*}   
for some nonnegative function $p_t^n(x,y)$ which is defined and smooth
on $(0, \infty)\times B_n(0)\times B_n(0)$.  Here we recall that $B_n(0)$ is the open ball of radius $n$ centered at the origin in $\RR^d$.  
Certainly, the transition kernel of $x_{t\wedge \tau_n}$ contains a singular component concentrated
on the boundary of $B_n(0)$. However, this is invisible to sets
contained in $B_n(0)$ since $B_n(0)$ is open.
\end{remark}

We now state the main general result of the paper.

\begin{theorem}\label{dsaimppos}\label{dsaimpposc}
  Suppose that $\mcC$ is $d$-dimensional and let $\{ y_1, \ldots,
  y_d\}\subset \mcC$ be a basis of $\mcC$ such that $\{y_1, \ldots,
  y_k \}\subset \mcC^{\emph{\text{o}}}$ and $\{ y_{k+1}, \ldots, y_d
  \}\subset \mcC^{\text{\emph{e}}}$.  For $x\in \RR^d$, define the set
\begin{align*}
  \mcD(x) &= \big \{ x\big\}+ \big\{\textstyle{\sum_{i=1}^k} \alpha_i
  y_i + \textstyle{\sum_{j=k+1}^{d}\lambda_j y_j}\, : \, \alpha_i \in
  \RR, \, \lambda_j >0\big\}.  
\end{align*} 
and suppose that $x, z\in \RR^d$ are such that $z\in \mcD(x)$. 
\begin{enumerate}[(a)]
\item\label{dsaimppos1} For all $T>0$ there exist $t\in (0,T)$ and $N\in  \NN$ such that 
\begin{equation*}
  p_{t}^n(x,z) >0 \,\, \text{ for all } \,\, n\geq N.   
\end{equation*}
\item\label{dsaimppos3} If there exists an equilibrium point $y\in
  \RR^{d}$ of $\mcG=~\{ X_0 +
  \textstyle{\sum}_{j=1}^r u_j X_j \, : \, u_j \in \RR\}$ such that
  $y\in \mcD(x)$ and $z\in \mcD(y)$, then for all $T>0$ there exists $N\in \NN$ such that 
\begin{equation*}
  p_t^n(x,z)>0\,\, \text{ for all }\,\, t\geq T, \, n\geq N.  
\end{equation*}
\end{enumerate}
\end{theorem}

\begin{remark}
\label{rem:noexp}
Suppose that $\mcC$ is $d$-dimensional and that $x_t$ is
\emph{non-explosive}; that is, for every $x\in \RR^d$
\begin{align*}
\PP_x \{ \tau_\infty < \infty\} =0.  
\end{align*}
Then $x_t$ has a probability density function $p_t(x,y)$ with respect
to Lebesgue measure $dy$ which is smooth on $(0, \infty) \times \RR^d
\times \RR^d$.  Moreover, all conclusions of Theorem \ref{dsaimppos}
hold with $p^n_t(x,z)$ replaced by $p_t(x,z)$.
\end{remark}

\begin{remark}
  Even if $\mcC$ is $d$-dimensional, it is still
  possible that the set $\mcD(x)$ cannot be chosen to be the entire
  space $\RR^d$.  See Example \ref{exx2y2} in Section
  \ref{sec:examples}.
\end{remark}

\begin{remark}
  It is worth emphasizing that $y \in \RR^d$ can be an equilibrium
  without being an equilibrium point of the drift vector field $X_0$. For
  example, if $X_0(y_1,y_2)= (g(y_1,y_2)(1-y_2), f(y_2,y_1) ) $ for some scalar functions $f,g$ and $X_1=(0,1)$  then all points of the form
  $(y_1,1)$ are equilibrium points since $X(y_1,1) + u X_1= (0,0)$ if $u=- f(y_1,1)$. 

\end{remark}

Using the results of Theorem \ref{dsaimppos}, we will also show:

\begin{theorem} 
\label{thm:posinv}
Suppose that $\mcC$ is $d$-dimensional and $x_t$ is non-explosive. Let $\mcD(x)$ be as in the statement of Theorem \ref{dsaimppos}.  Then there is
  at most one invariant probability measure corresponding to the Markov process $x_t$ defined by \eqref{eqn:sde1}. Moreover, if such an invariant probability measure $\mu$ exists, then $\mu(dx)=m(x)\, dx$ for some smooth,
  non-negative function $m$ and if $x \in \supp(\mu)$ then for all $z\in \mcD(x)$, $m(z)> 0$. 

\end{theorem}



\section{Examples}
\label{sec:examples}
Before proving the main results, we apply them to specific examples to
show their utility.  A ``non-example", that is an example where
Theorem \ref{dsaimppos} is not applicable, is given in the next
section in Remark \ref{nonExample} as it fits in better with the
discussion there.

  \begin{example}
As a first example, we consider the Langevin dynamics on $\RR^{2d}$, $d\geq 1$, 
\begin{equation}\label{eqn:langevin}
  \begin{aligned}
    d x_t &= [-\gamma x_t - \nabla F(y_t) ]\, dt + \sum_{j=1}^{d} \sigma_j \, dW^{j}_t\\
    dy_t &= x_t \, dt     
  \end{aligned}
\end{equation}
where $x_t, y_t \in \RR^{d}$, $\gamma >0$ is a constant, $F\in
C^\infty(\RR^d: \RR)$, $\sigma_j \in \RR^d$ and the $W_t^{j}$ are
independent standard Wiener processes.  So that solutions to
\eqref{eqn:langevin} do not explode in finite time, we assume that $F$
satisfies the one-sided Lipschitz condition and concavity and growth
assumptions of Condition 3.1 of \cite{MSH}.  A prototypic example
of a potential which satisfies these assumptions is  $F(y)=\tfrac14|y|^4-\tfrac12|y|^2$.

As a consequence of Theorem \ref{dsaimppos}, we now prove: 
\begin{corollary}
\label{cor:langevin}
If $\text{\emph{span}}\{\sigma_1, \ldots, \sigma_d\}=\RR^d$, then for
all $(x,y), (x',y') \in \RR^{2d}$ and $t>0$
\begin{align*}
p_{t}((x,y), (x',y') ) >0.  
\end{align*}
\end{corollary}
\begin{proof}
  Let $\boldsymbol{0}=(0,0,\ldots, 0)\in \RR^d$ and let $\mcG=\{ X_0 +
  \sum_{j=1}^d u_j X_j\, : \, u_j \in \RR\}$ where
\begin{align*}
X_0(x,y) = 
 \left( \begin{array}{c}
-\gamma x + \nabla F(y)\\
x \end{array} \right) \qquad \text{and} \qquad X_j(x,y) = \left( \begin{array}{c}
\sigma_j \\
 \textbf{0}\end{array} \right).   
\end{align*}
We begin by computing $\mathcal{C}$ (defined in the introduction) corresponding to equation \eqref{eqn:langevin}.  Since $n( X_0, X_j)=1$ for all $j$, we see that  
\begin{align*}
\mcG_{1}^{\text{o}} &\supset \{[X_j, X_0]\,: \, j=1,2, \ldots, d\}
\end{align*}          
and 
\begin{align*}
[X_j, X_0] (x,y)=  \left( \begin{array}{c}
-\gamma \sigma_j \\
\sigma_j \end{array} \right) .
\end{align*}
Hence, in particular, $\mathcal{C} \supset \{ X_j, [X_j, X_0]\, : \, j=1,2, \ldots, d\}$.  Since the vectors $\sigma_1, \ldots, \sigma_d$ are linearly independent, it follows that $\mathcal{C}$ has a basis.  Additionally, since $\mathcal{C}^{\text{o}} \supset  \{ X_j, [X_j, X_0]\, : \, j=1,2, \ldots, d\}$ we can choose a basis so that $\mathcal{D}(x,y)=\RR^{2d}$ for all $(x,y) \in \RR^{2d}$.  To finish proving the result, we claim that the origin $(\boldsymbol{0}, \boldsymbol{0})\in \RR^{2d}$ is an equilibrium point of $\mcG$.  Indeed, since
\begin{align*}
X_0(\boldsymbol{0}, \boldsymbol{0}) + \sum_{j=1}^d  u_j X_j(\boldsymbol{0}, \boldsymbol{0}) = \left( \begin{array}{c}
-\nabla F(\boldsymbol{0})\\
\boldsymbol{0} \end{array} \right) +   \left( \begin{array}{c}
\sum_{j=1}^d u_j \sigma_j \\
\boldsymbol{0} \end{array} \right) 
\end{align*}
and the $\sigma_j$ form a basis of $\RR^d$, we may choose real numbers $u_j \in \RR$ such that  
\begin{align*}
X_0(\boldsymbol{0}, \boldsymbol{0}) + \sum_{j=1}^d  u_j X_j(\boldsymbol{0}, \boldsymbol{0}) =  \left( \begin{array}{c}
\boldsymbol{0}\\
\boldsymbol{0} \end{array} \right). 
\end{align*}
In light of Remark \ref{rem:noexp}, applying Theorem ~\ref{dsaimppos}~(\ref{dsaimppos3}) finishes the proof of Corollary \ref{cor:langevin}. 
\end{proof}
  
 \end{example}


\begin{example}\label{exx2y2} Let $a_{1}, a_{2} \in \RR$,
$\alpha_2>\alpha_{1}>0$, and $\epsilon>0$.  With motivations from
turbulent transport of inertial particles, the stochastic differential equation on $\RR^{2}$ given by
\begin{equation}
\begin{aligned}
\label{eqn:sdeeuc} 
 dx_t&=(a_1 x_t -\alpha_1 x_t^2+y_t^2) \, dt \\ 
 dy_t &= (a_{2} y_t -\alpha_{2} x_t y_t ) \, dt + \epsilon \, dW^2_t
\end{aligned}
\end{equation} 
is considered in \cite{BHW}.  Here, we strengthen the results of
Section~4 of this work. A more hands on application of some of the
ideas of this note were applied to a specific case of this example in
Section~11 of \cite{AKM}.  Applying Theorem~2.1 of \cite{BHW}, we first
note that $(x_t, y_t)$ is non-explosive.

We now prove:
\begin{corollary}
\label{cor:2deuc}
Suppose that $(x,y)\in \RR^2$ satisfies 
\begin{align*}
x < \frac{a_1-|a_{1}|}{2\alpha_1}\,\text{ or }\,x\geq
\frac{a_1+|a_{1}|}{2\alpha_1}.
\end{align*} 
Then for all $t>0$ and $(x',y')\in \RR^2$ with $x' > x$
\begin{align*}
p_{t}((x,y), (x', y'))>0.
\end{align*}
Otherwise if $(x,y)\in \RR^2$ satisfies 
\begin{align*}
\frac{a_1 - |a_{1}|}{ 2\alpha_1 }\leq x \leq \frac{a_1+ |a_{1}|}{ 2\alpha_1 }, 
\end{align*}
then for all $t>0$ and $(x',y') \in \RR^2$ with $x' > \frac{a_1 + |a_1|}{2\alpha_1}$
\begin{equation*} 
p_{t}((x,y), (x',y')) >0.
\end{equation*}
\end{corollary}

\begin{remark}
\label{rem:notoptimal}
  It is important to point out that Corollary~\ref{cor:2deuc} is not sharp. For example if $a_1=a_2=0$, $\alpha_1=1$ and $\alpha_2=2$, it was shown in Section 11 of \cite{AKM} that, in addition to the result above, for all $(x,y), (x',y')\in \RR^2$ with $x'>0$ 
  $$p_{t}((x,y), (x', y'))>0 $$ for all $t>0$ sufficiently large. The
  weakness of our result is due to the fact that Theorem~\ref{dsaimpposc} does not fully exploit the flow along $X_0$ in favor of making general statements for any positive time. However, Corollary~\ref{cor:2deuc} is more than sufficient to prove unique ergodicity in equation \eqref{eqn:sdeeuc}.  Nevertheless, it is not hard to bootstrap from Corollary~\ref{cor:2deuc} to obtain the full (sharp) result proved in \cite{AKM}.
\end{remark}

\begin{proof}
As in the previous example, we begin by computing the set $\mcC$ corresponding to equation \eqref{eqn:sdeeuc}.  Let 
\begin{equation*}
 \mcG= \big\{X_0 + u  X_1 \, : \, u \in
\RR \big\}
\end{equation*}
where $X_0=(a_1 x- \alpha_1 x^{2} +
y^2) \partial_{x} + (a_2 y - \alpha_2 x y)\partial_{y}$ and $X_1
= \partial_y$.  Since $n(X_0, X_1)=2$, we find that $\text{ad}^{2}X_1(X_0)=2\partial_{x} \in \mcG_{1}^{\text{e}}$.  Let   
\begin{equation*} \mathcal{D}(x,y) = \{(x, y) + u (0,1) + \lambda
(1,0) \, : \, u \in \RR, \, \lambda > 0 \}.
\end{equation*} 
As opposed to the previous example, the set $\mcD(x,y)$ is not the entire space.  Hence we must make sure we have enough equilibrium points in the right locations.

Consider the polynomial equation
\begin{align*}
a_1 x- \alpha_1 x^2 + y^2 &=0\\
a_2 y -\alpha_2 x y + u &= 0  
\end{align*}
where $u \in \RR$.  Clearly, any pair $(x,y)\in \RR^2$ satisfying the above equations for some $u\in \RR$ is an equilibrium point of $\mathcal{G}$.  In particular, we may solve $a_1 x - \alpha_1 x^2 + y^2 =0$ producing 
\begin{align*}
x= \frac{a_1 \pm \sqrt{a_1^2 + 4 \alpha_1 y^2}}{2\alpha_1}.  
\end{align*}    
Since we may pick $u=\alpha_2 xy - a_2 y$, we therefore deduce that all points $(x,y)\in \RR^{2}$ such that either
\begin{equation*} x \geq \frac{a_{1} + |a_{1}|}{2\alpha_1} \qquad
\text{ or } \qquad x\leq \frac{a_{1}-|a_{1}|}{2\alpha_1}
\end{equation*} are equilibrium points for the control system $\mcG$.  Hence Remark \ref{rem:noexp} now implies Corollary \ref{cor:2deuc}.   
\end{proof}

\end{example}


\begin{example}\label{exburg}

Let $\nu>0$ be a constant.  We now study Galerkin truncations of the following randomly forced two-dimensional viscous Burgers' equation
\begin{equation}\label{burgeq}
\partial_{t} u(\mathbf{x}, t )+ (u(\mathbf{x},t) \cdot
\nabla_{\mathbf{x}}) u( \mathbf{x}, t)= \nu \Delta_{\mathbf{x}} u(\mathbf{x}, t
) + \xi(\mathbf{x}, t)
\end{equation} 
with periodic boundary conditions on the torus $\mathbb{T}^2=[0,
2\pi]^2$.  Here, we assume that there is no mean flow and that $\xi$
is a Gaussian process which is white in time and colored in space.
To emphasize, we do not require the divergence free condition
$\nabla \cdot u =0$; hence, \eqref{burgeq} is not the 2D Navier Stokes
equation. Moreover, we do not restrict ourselves to gradient solutions as is often done when considering the multidimensional Burgers equation.  In the dynamics \eqref{burgeq}, we are precisely interested how the divergence free forcing spreads to the non-divergence free (gradiant-like directions). Since one does not have global solutions in this setting, here we must make use of the stopped processes.

Let us now be more precise. Writing
\begin{align*}
\sum_{0\neq \mathbf{k} \in
\ZZ^{2}} u_{\mathbf{k}}(t) e^{-i\langle \mathbf{k}, \mathbf{x}\rangle},
\end{align*}
where $\langle \cdot, \cdot \rangle$ denotes the dot product, and fixing a positive integer $N\geq 2$, we consider the following stochastic differential equation on $\CC^{2((2N+1)^2-1)}$
\begin{eqnarray}\label{eqn:sburgmodes} d u_{\mathbf{k}} &=& [i
F^{N}_{\mathbf{k}}(u) - \nu | \mathbf{k}|^{2}u_{\mathbf{k}}]\, dt
+\frac{\mathbf{k}^{\perp}}{|\mathbf{k}|^{2}}( \sigma_{\mathbf{k}}\,
dB^{\mathbf{k}, (1)}_t + i \sigma_{\mathbf{k}}' \, dB^{\mathbf{k},
(2)}_t) \\ \nonumber &\,&\, \, \, \, \, + \,
\frac{\mathbf{k}}{|\mathbf{k}|^{2}}( \gamma_{\mathbf{k}}\,
dW^{\mathbf{k}, (1)}_t + i \gamma_{\mathbf{k}}' \, dW^{\mathbf{k},
(2)}_t)
\end{eqnarray}
where 
\begin{itemize}
\item $u_\mathbf{k} \in \CC^2$;
\item the equation is over all indices $\mathbf{k} \in
H_{N}=\big\{\mathbf{k} \in \ZZ^{2}\setminus \{(0,0)\}\, : \,
\|\mathbf{k} \|_{\infty}\leq N \big\}$;
\item  \begin{equation*} F_{\mathbf{k}}^{N}(u)=\sum_{\mathbf{l},
\mathbf{k}-\mathbf{l}\in H_{N}} \langle u_{\mathbf{l}},
\mathbf{k}-\mathbf{l} \rangle u_{\mathbf{k}-\mathbf{l}};
\end{equation*}
\item $\sigma_{\mathbf{k}}, \sigma_{\mathbf{k}}',
\gamma_{\mathbf{k}}, \gamma_{\mathbf{k}}'\in \RR$;
\item $\mathbf{k}^{\perp} =(k_1, k_2)^{\perp}=(-k_2, k_1)$; 
\item $\{
B^{\mathbf{k}, (1)}_t , B^{\mathbf{k}, (2)}_t , W^{\mathbf{k}, (1)}_t
, W^{\mathbf{k}, (2)}_t\}_{\mathbf{k} \in H_{N}}$ is a set of
independent Brownian motions.  
\end{itemize}


To further illuminate the discussion, we first split the equation into incompressible and compressible directions. To this end, write  
\begin{eqnarray*} u_{\mathbf{k}}&=&
w_{\mathbf{k}}\frac{\mathbf{k}^{\perp}}{|\mathbf{k}|^{2}}+
q_{\mathbf{k}}\frac{\mathbf{k}}{|\mathbf{k}|^{2}}\\
F_{\mathbf{k}}(u)&=& F_{\mathbf{k}}^{\perp} (w, q)
\frac{\mathbf{k}^{\perp}}{|\mathbf{k}|^{2}}+
F_{\mathbf{k}}^{\parallel}(w, q) \frac{\mathbf{k}}{|\mathbf{k}|^{2}}
\end{eqnarray*} where $w_{\mathbf{k}}, q_{\mathbf{k}}\in \CC$.  In particular, equation \eqref{eqn:sburgmodes} now becomes
\begin{eqnarray}\label{eqn:burgcsde} dw_{\mathbf{k}}&=&
[-\nu|\mathbf{k}|^{2} w_{\mathbf{k}} + i
F_{\mathbf{k}}^{\perp}(w,q)]\, dt +\sigma_{\mathbf{k}}\,
dB^{\mathbf{k}, (1)}_t + i \sigma_{\mathbf{k}}' \, dB^{\mathbf{k},
(2)}_t\\ \nonumber dq_{\mathbf{k}}&=&[-\nu
|\mathbf{k}|^{2}q_{\mathbf{k}}+ i F_{\mathbf{k}}^{\parallel}(w,q)]\,
dt + \gamma_{\mathbf{k}} \, dW^{\mathbf{k}, (1)}_t + i
\gamma_{\mathbf{k}}' \, dW^{\mathbf{k}, (2)}_{t} 
\end{eqnarray} 
for some $F_\mathbf{k}^\perp$, $F_\mathbf{k}^\parallel$ to be computed in a moment.  Note that \eqref{eqn:burgcsde} evolves on $\CC^{2((2N+1)^2-1)}=\CC^{8N(N+1)}$ for all $t< \tau_\infty$.   

We will now use Theorem \ref{dsaimppos} to prove the following result:
\begin{theorem}  
\label{thm:burg}
Suppose that $$\{ \mathbf{k} \in H_N \, : \, \sigma_\mathbf{k} \neq 0, \sigma_{\mathbf{k}}' \neq 0\} \supset \{\mathbf{k} \in H_{N} \, :
\, \| \mathbf{k} \|_{\infty}=1 \}.$$  Then for all $(w, q), (w', q') \in \CC^{8N(N+1)}$ and $T>0$, there exists $N\in \NN$ large enough so that 
\begin{align*}
p_t^n ((w,q), (w',q'))>0\,\, \text{ for all } t\geq T, \, n\geq N.    
\end{align*} 
\end{theorem}

\begin{remark}
It is interesting to note that, even if the process $(w_t, q_t)$ is assumed to be incompressible initially; that is, $(w_0, q_0)=(w,0)\in \CC^{8N(N+1)}$, a small amount of low mode forcing ensures that any mixture of incompressible and compressible states becomes instantaneously possible.  As we will see in the proof below, this cannot happen if we do not force the incompressible directions.  In particular, if we assume that the process $(w_t, q_t)$ is initially compressible; that is, $(w_0, q_0)=(0,q)$ and $\sigma_\mathbf{k} = \sigma_\mathbf{k}' =0$ for all $\mathbf{k}\in H_N$, then $w_t \equiv 0$ for all $t\geq 0$.         
\end{remark}

\begin{proof}[Proof of Theorem \ref{thm:burg}]
We will first write out and symmetrize the nonlinear terms $F_\mathbf{k}^\perp$ and $F_\mathbf{k}^\parallel$.  Using the relations $\langle \mathbf{k}^{\perp}, \mathbf{l}
\rangle=-\langle \mathbf{k}, \mathbf{l}^{\perp} \rangle$ and $\langle
\mathbf{k}^{\perp}, \mathbf{l}^{\perp}\rangle =\langle \mathbf{k},
\mathbf{l} \rangle$, we find that
\begin{eqnarray*} F_{\mathbf{k}}^{\perp}(w, q) &=& \sum_{\mathbf{l}, \, \mathbf{k}-\mathbf{l}\in
H_{N}} w_{\mathbf{l}}w_{\mathbf{k}-\mathbf{l}}\frac{\langle
\mathbf{l}^{\perp}, \mathbf{k} \rangle \langle \mathbf{k}-\mathbf{l},
\mathbf{k}\rangle}{|\mathbf{l}|^{2}|\mathbf{k}-\mathbf{l}|^{2}} +
w_{\mathbf{l}}q_{\mathbf{k}-\mathbf{l}}\frac{\langle
\mathbf{l}^{\perp}, \mathbf{k}
\rangle^{2}}{|\mathbf{l}|^{2}|\mathbf{k} -\mathbf{l}|^{2}}\\ &\, & \,
\, + \sum_{\mathbf{l}, \, \mathbf{k}-\mathbf{l}\in
H_{N}}
q_{\mathbf{l}}w_{\mathbf{k}-\mathbf{l}}\frac{\langle \mathbf{l},
\mathbf{k}-\mathbf{l} \rangle \langle \mathbf{k}-\mathbf{l},
\mathbf{k}\rangle}{|\mathbf{l}|^{2}|\mathbf{k}-\mathbf{l}|^{2}} -
q_{\mathbf{l}}q_{\mathbf{k}-\mathbf{l}}\frac{\langle \mathbf{l},
\mathbf{k}-\mathbf{l} \rangle \langle \mathbf{l},
\mathbf{k}^{\perp}\rangle }{|\mathbf{l}|^{2}|\mathbf{k}
-\mathbf{l}|^{2}}
\end{eqnarray*} and
\begin{eqnarray*} F_{\mathbf{k}}^{\parallel}(w, q) &=&
\sum_{\mathbf{l}, \, \mathbf{k}-\mathbf{l}\in
H_{N}}
-w_{\mathbf{l}}w_{\mathbf{k}-\mathbf{l}}\frac{\langle
\mathbf{l}^{\perp}, \mathbf{k}
\rangle^{2}}{|\mathbf{l}|^{2}|\mathbf{k}-\mathbf{l}|^{2}} +
w_{\mathbf{l}}q_{\mathbf{k}-\mathbf{l}}\frac{\langle
\mathbf{l}^{\perp}, \mathbf{k} \rangle \langle \mathbf{k}-\mathbf{l},
\mathbf{k} \rangle}{|\mathbf{l}|^{2}|\mathbf{k} -\mathbf{l}|^{2}}\\
&\, & \, \, \sum_{\mathbf{l}, \, \mathbf{k}-\mathbf{l}\in
H_{N}}-
q_{\mathbf{l}}w_{\mathbf{k}-\mathbf{l}}\frac{\langle \mathbf{l},
\mathbf{k}-\mathbf{l} \rangle \langle \mathbf{l}^{\perp},
\mathbf{k}\rangle}{|\mathbf{l}|^{2}|\mathbf{k}-\mathbf{l}|^{2}} +
q_{\mathbf{l}}q_{\mathbf{k}-\mathbf{l}}\frac{\langle \mathbf{l},
\mathbf{k}-\mathbf{l} \rangle \langle \mathbf{k}-\mathbf{l},
\mathbf{k}\rangle }{|\mathbf{l}|^{2}|\mathbf{k} -\mathbf{l}|^{2}}.
\end{eqnarray*} After considering the effect of the mapping $(\mathbf{l},
\mathbf{k}-\mathbf{l})\mapsto (\mathbf{k}-\mathbf{l}, \mathbf{l})$ on each of the terms above, we may write
\begin{eqnarray*} F_{\mathbf{k}}^{\perp}(w,q)&=&\sum_{\mathbf{l}, \, \mathbf{k}-\mathbf{l}\in
H_{N}} w_{\mathbf{l}}w_{\mathbf{k}-\mathbf{l}}\frac{\langle
\mathbf{l}^{\perp}, \mathbf{k} \rangle}{2}\bigg(
\frac{1}{|\mathbf{l}|^{2}}-\frac{1}{|\mathbf{k}-\mathbf{l}|^{2}}\bigg)+
w_{\mathbf{l}}q_{\mathbf{k}-\mathbf{l}} \frac{\langle
\mathbf{k}-\mathbf{l}, \mathbf{k}\rangle
}{|\mathbf{k}-\mathbf{l}|^{2}}\\ F_{\mathbf{k}}^{\parallel}(w, q)
&=&\sum_{\mathbf{l}, \, \mathbf{k}-\mathbf{l}\in
H_{N}}-w_{\mathbf{l}}w_{\mathbf{k}-\mathbf{l}}\frac{\langle
\mathbf{l}^{\perp}, \mathbf{k}
\rangle^{2}}{|\mathbf{l}|^{2}|\mathbf{k}-\mathbf{l}|^{2}}+
w_{\mathbf{l}}q_{\mathbf{k}-\mathbf{l}}\frac{\langle
\mathbf{l}^{\perp}, \mathbf{k} \rangle \langle \mathbf{k}-\mathbf{l},
\mathbf{k} +\mathbf{l}\rangle}{|\mathbf{l}|^{2}|\mathbf{k}
-\mathbf{l}|^{2}}\\ &\, & \, \,+ \sum_{\mathbf{l}, \, \mathbf{k}-\mathbf{l}\in
H_{N}}q_{\mathbf{l}}q_{\mathbf{k}-\mathbf{l}} \frac{\langle\mathbf{l},
\mathbf{k} -\mathbf{l}
\rangle}{2}\frac{|\mathbf{k}|^{2}}{|\mathbf{l}|^{2}|\mathbf{k}-\mathbf{l}|^{2}}.  
\end{eqnarray*}   
The assertion made in the previous remark now follows easily from these expressions since if $\sigma_\mathbf{k}=\sigma_\mathbf{k}'=0$ for all $\mathbf{k} \in H_N$ and $w_0=0$, then $w_t=(w_\mathbf{k}(t))_{\mathbf{k}\in H_N}\equiv0$ for all times $t$.

To prove Theorem \ref{thm:burg}, we do as in the previous two examples and start by computing $\mcC$ corresponding to \eqref{eqn:burgcsde}.  Define
\begin{equation*} \mcG= \bigg\{ X_0 + \sum_{\mathbf{k} \in
FD_{I}} u_{\mathbf{k}} X_{\mathbf{k}} + v_{\mathbf{k}}Y_{\mathbf{k}}\,
: \, u_{\mathbf{k}}, v_{\mathbf{k}}\in \RR \bigg\}
\end{equation*} where
\begin{eqnarray*} X_{0} &=& \sum_{\mathbf{k}\in H_{N}}\big[-\nu
|\mathbf{k}|^{2} w_{\mathbf{k}} + i F^{\perp}_{\mathbf{k}}(w, q)
\big]\frac{\partial}{\partial w_{\mathbf{k}}}+ \big[-\nu
|\mathbf{k}|^{2} q_{\mathbf{k}} + i F^{\parallel}_{\mathbf{k}}(w,
q) \big]\frac{\partial}{\partial q_{\mathbf{k}}}\\ &\, & \, \, +
\sum_{\mathbf{k}\in G_{N}}\big[-\nu
|\mathbf{k}|^{2}\bar{w}_{\mathbf{k}} - i F^{\perp}_{\mathbf{k}}(\bar{w}, \bar{q}) \big]\frac{\partial}{\partial
\bar{w}_{\mathbf{k}}}+ \big[-\nu |\mathbf{k}|^{2}\bar{ q}_{\mathbf{k}}
- i F^{\parallel}_{\mathbf{k}}(\bar{w},\bar{ q})
\big]\frac{\partial}{\partial \bar{q}_{\mathbf{k}}}
\end{eqnarray*} and
\begin{equation*} X_{\mathbf{k}}= \frac{\partial}{\partial
w_{\mathbf{k}}}+ \frac{\partial}{\partial \bar{w}_{\mathbf{k}}},
\qquad \qquad Y_{\mathbf{k}}= i \frac{\partial}{\partial
w_{\mathbf{k}}} - i \frac{\partial}{\partial \bar{w}_{\mathbf{k}}}.
\end{equation*} 
Notice that $n(X_{0}, X_{\mathbf{j}})=1$ for all $\mathbf{j} \in \{\mathbf{k} \in H_N \, : \, \sigma_\mathbf{k}\neq 0, \, \sigma_{\mathbf{k}}' \neq 0 \}$ since there are no diagonal terms in the nonlinear part of $X_0$. In particular, $$[X_{\mathbf{j}}, X_{0}] \in \mcG_{1}^{ \text{o}}\,\text{ for all }\,\mathbf{j} \in\{\mathbf{k} \in H_N \, : \, \sigma_\mathbf{k}\neq 0, \, \sigma_{\mathbf{k}}' \neq 0 \}. $$  Moreover, one can compute these commutators to see that 
\begin{eqnarray*} [X_{\mathbf{j}},X_0]&=&-\nu
  |\mathbf{j}|^{2}\frac{\partial}{\partial w_{\mathbf{j}}} -\nu
  |\mathbf{j}|^{2}\frac{\partial}{\partial \bar{w}_{\mathbf{j}}} \\
  &&\, \, \, + \, i \sum_{\mathbf{k} \in H_{N}}
  \bigg[w_{\mathbf{k}-\mathbf{j}}\langle \mathbf{j}^{\perp},
  \mathbf{k} \rangle\bigg(
  \frac{1}{|\mathbf{j}|^{2}}-\frac{1}{|\mathbf{k}-\mathbf{j}|^{2}}\bigg)+\,
  q_{\mathbf{k}-\mathbf{j}}\frac{\langle \mathbf{k}- \mathbf{j},
    \mathbf{k}\rangle}{|\mathbf{k}-\mathbf{j}|^{2}}\bigg]
  \frac{\partial}{\partial w_{\mathbf{k}}}\\ &&\, \, \, -\, i
  \sum_{\mathbf{k} \in H_{N}}
  \bigg[\bar{w}_{\mathbf{k}-\mathbf{j}}\langle \mathbf{j}^{\perp},
  \mathbf{k} \rangle\bigg(
  \frac{1}{|\mathbf{j}|^{2}}-\frac{1}{|\mathbf{k}-\mathbf{j}|^{2}}\bigg)+\bar{q}_{\mathbf{k}-\mathbf{j}}\frac{\langle
    \mathbf{k}- \mathbf{j},
    \mathbf{k}\rangle}{|\mathbf{k}-\mathbf{j}|^{2}}\bigg]
  \frac{\partial}{\partial \bar{w}_{\mathbf{k}}}\\ &&\, \, \, + \, i
  \sum_{\mathbf{k} \in H_{N}}\bigg[ -2w_{\mathbf{k}-\mathbf{j}}
  \frac{\langle \mathbf{j}^{\perp},
    \mathbf{k}\rangle^{2}}{|\mathbf{j}|^{2}|\mathbf{k}-\mathbf{j}|^{2}}+q_{\mathbf{k}-\mathbf{j}}\frac{\langle
    \mathbf{j}^{\perp}, \mathbf{k}\rangle \langle \mathbf{k}-
    \mathbf{j}, \mathbf{k} +
    \mathbf{j}\rangle}{|\mathbf{j}|^{2}|\mathbf{k}-\mathbf{j}|^{2}}\bigg]\frac{\partial}{\partial
    q_{\mathbf{k}}}\\ &&\, \, \, - \, i \sum_{\mathbf{k} \in
    H_{N}}\bigg[ -2\bar{w}_{\mathbf{k}-\mathbf{j}} \frac{\langle
    \mathbf{j}^{\perp},
    \mathbf{k}\rangle^{2}}{|\mathbf{j}|^{2}|\mathbf{k}-\mathbf{j}|^{2}}+\bar{q}_{\mathbf{k}-\mathbf{j}}\frac{\langle
    \mathbf{j}^{\perp}, \mathbf{k}\rangle \langle \mathbf{k}-
    \mathbf{j}, \mathbf{k} +
    \mathbf{j}\rangle}{|\mathbf{j}|^{2}|\mathbf{k}-\mathbf{j}|^{2}}\bigg]\frac{\partial}{\partial
    \bar{q}_{\mathbf{k}}}.
\end{eqnarray*}  
Note also that for all $\mathbf{j}, \mathbf{m} \in
\{\mathbf{k} \in H_N \, : \, \sigma_\mathbf{k}\neq 0, \, \sigma_{\mathbf{k}}' \neq 0 \}$ such that $\mathbf{j} + \mathbf{m} \in H_{N}$
\begin{equation*} n(X_{\mathbf{m}}, [X_{\mathbf{j}}, X_0])=
n(Y_{\mathbf{m}}, [X_{\mathbf{j}}, X_0] )=1.
\end{equation*} Hence for all $\mathbf{j}, \mathbf{m} \in \{\mathbf{k} \in H_N \, : \, \sigma_\mathbf{k}\neq 0, \, \sigma_{\mathbf{k}}' \neq 0 \}$ with
$\mathbf{j}+ \mathbf{m} \in H_{N}$, $[X_{\mathbf{m}}, [X_{\mathbf{j}},
X_0]]\in \mcG_{2}^{ \text{o}}$ and $
[Y_{\mathbf{m}},[X_{\mathbf{j}}, X_0]]\in \mcG_{2}^{ \text{o}}$.
Computing these commutators we find that  
\begin{equation}\label{blie1} [X_{\mathbf{m}}, [X_{\mathbf{j}},
X_0]]= \langle \mathbf{j}^{\perp}, \mathbf{m}\rangle\bigg(
\frac{1}{|\mathbf{j}|^{2}}-
\frac{1}{|\mathbf{m}|^{2}}\bigg)Y_{\mathbf{j} + \mathbf{m}}-
2\frac{\langle\mathbf{j}^{\perp}, \mathbf{m}
\rangle^{2}}{|\mathbf{j}|^{2}|\mathbf{m}|^{2}}
\widetilde{Y}_{\mathbf{j} + \mathbf{m}}
\end{equation} and
\begin{equation}\label{blie2} [Y_{\mathbf{m}}, [X_{\mathbf{j}},
X_0]]=- \langle \mathbf{j}^{\perp}, \mathbf{m}\rangle\bigg(
\frac{1}{|\mathbf{j}|^{2}}-
\frac{1}{|\mathbf{m}|^{2}}\bigg)X_{\mathbf{j} + \mathbf{m}}+
2\frac{\langle\mathbf{j}^{\perp}, \mathbf{m}
\rangle^{2}}{|\mathbf{j}|^{2}|\mathbf{m}|^{2}}
\widetilde{X}_{\mathbf{j} + \mathbf{m}}
\end{equation} where
\begin{equation*} \widetilde{X}_{\cdot}= \frac{\partial}{\partial
q_{\cdot}}+ \frac{\partial}{\partial \bar{q}_{\cdot}},\qquad
\widetilde{Y}_{\cdot}=i \frac{\partial}{\partial q_{\cdot}}- i
\frac{\partial}{\partial \bar{q}_{\cdot}}.
\end{equation*}
We will now use the above computations to prove that 
\begin{align*}
\big\{ X_\mathbf{j}, Y_\mathbf{j}, \widetilde{X}_\mathbf{j}, \widetilde{Y}_\mathbf{j} \, : \, \| \mathbf{k}\|_\infty \leq k\}\subset \mcC^\text{o} 
\end{align*}
for all $k=1,2, \ldots, N$ by induction on $k$.  It will then follow that $\mcC^\text{o}$ spans the tangent space, and so we may pick $\mcD(w,q)=\CC^{8N(N+1)}$ for all $(w, q) \in \CC^{8N(N+1)}$.

To prove the claim when $k=1$, first substitute 
\begin{align*}
(\mathbf{j},
\mathbf{m})=(( 1,0), (0, 1)), ((1,0), (0,-1)), ((-1,0), (0, -1)), ((-1, 0), (0, 1)) 
\end{align*}
into equations \eqref{blie1}-\eqref{blie2} to see
that $\widetilde{X}_{(1,1)}$, $\widetilde{Y}_{(1,1)}$, $
\widetilde{X}_{(1,-1)}$, $\widetilde{Y}_{(1,-1)}$,
$\widetilde{X}_{(-1,-1)}$, $\widetilde{Y}_{(-1,-1)},
\widetilde{X}_{(-1,1)}$, $\widetilde{Y}_{(-1,1)}\in \mcC^{\text{o}}$.
Substituting 
\begin{align*}
(\mathbf{j}, \mathbf{m})=(( 1,1), (0, - 1)), ((1,1),
(-1,0)), ((-1,1), (0, -1)), ((-1, -1), (1, 0)) 
\end{align*}
into the same equations and using the fact that $ X_{\mathbf{k}},  Y_{\mathbf{k}}\in
\mcC^\text{o}$ for any $\| \mathbf{k}\|_{\infty}=1$, we find by taking linear combinations that $\widetilde{X}_{(1,0)}$,
$\widetilde{Y}_{(1,0)}$, $ \widetilde{X}_{(0,1)}$,
$\widetilde{Y}_{(0,1)}$, $\widetilde{X}_{(-1,0)}$,
$\widetilde{Y}_{(-1,0)}, \widetilde{X}_{(0,-1)}$,
$\widetilde{Y}_{(0,-1)}\in \mcC^{\text{o}}$.  This proves the initial statement in the inductive argument.  Suppose now that for some $1\leq k< N$ 
\begin{equation*} \big\{X_{\mathbf{j}}, Y_{\mathbf{j}},
\widetilde{X}_{\mathbf{j}},
\widetilde{Y}_{\mathbf{j}}\, : \, \mathbf{j} \in H_{N},
\|\mathbf{j}\|_{\infty}\leq k \big\}  \subset \mcC^\text{o}.  
\end{equation*} 
Note that if $\mathbf{m}, \mathbf{j}\in
H_{N}$ are such that $\|\mathbf{m}\|_{\infty}\leq k$, $\|\mathbf{j}
\|_{\infty}=1$, then
$[\widetilde{X}_{\mathbf{m}}, [X_{\mathbf{j}}, X_{0}]]\in \mcC^\text{odd} $ and $[\widetilde{Y}_{\mathbf{m}}, [X_{\mathbf{j}}, X_0
]]\in \mcC^\text{o}$.  Note moreover that
\begin{equation}\label{blie3} [\widetilde{X}_{m}, [X_{\mathbf{j}},
X_{0}]]= \frac{\langle \mathbf{m}, \mathbf{j} + \mathbf{m}\rangle
}{|\mathbf{m}|^{2}}Y_{\mathbf{j} + \mathbf{m}}+ \frac{\langle
\mathbf{j}^{\perp}, \mathbf{m} \rangle \langle \mathbf{m}, \mathbf{m}
+
2\mathbf{j}\rangle}{|\mathbf{j}|^{2}|\mathbf{m}|^{2}}\widetilde{Y}_{\mathbf{j}+
\mathbf{m}}
\end{equation} and
\begin{equation}\label{blie4} [\widetilde{Y}_{m}, [X_{\mathbf{j}},
X_{0}]]=- \frac{\langle \mathbf{m}, \mathbf{j} + \mathbf{m}\rangle
}{|\mathbf{m}|^{2}}X_{\mathbf{j}+ \mathbf{m}}- \frac{\langle
\mathbf{j}^{\perp}, \mathbf{m} \rangle \langle \mathbf{m}, \mathbf{m}
+ 2\mathbf{j}\rangle}{|\mathbf{j}|^{2}|\mathbf{m}|^{2}}\widetilde{X}_{
\mathbf{j}+ \mathbf{m}}.
\end{equation} We claim that if $\mathbf{m}, \mathbf{j} \in H_{N}$ are
such that $|\mathbf{j}| \neq | \mathbf{m}|$ and $\langle
\mathbf{j}^{\perp}, \mathbf{m}\rangle \neq 0$, then the pairs
\eqref{blie1} and \eqref{blie3}, \eqref{blie2} and \eqref{blie4}, are
independent.  Indeed, if they are dependent under these assumptions,
then
\begin{equation*}
|\mathbf{j}|^{2}\langle \mathbf{m}, \mathbf{m} + \mathbf{j} \rangle
=\frac{1}{2} (|\mathbf{j}|^{2}-|\mathbf{m}|^{2})\langle \mathbf{m},
\mathbf{m} + 2\mathbf{j}\rangle
\end{equation*} which is true if and only if
\begin{equation*}
|\mathbf{j}|^{2}+|\mathbf{m}|^{2}+2 \langle \mathbf{m},
\mathbf{j}\rangle=0.
\end{equation*} Note that this equality is impossible since $|\mathbf{j}|\neq |\mathbf{m}|$.  Therefore, to finish the inductive argument, it suffices to show that 
for all $\mathbf{k}\in H_{N}$ with $\|
\mathbf{k}\|_{\infty}=k+1$, there exist $\mathbf{m}, \mathbf{j} \in
H_{N}$ such that 
\begin{itemize}
\item $\mathbf{m}+ \mathbf{j}=\mathbf{k}$;
\item $\|\mathbf{m} \|_{\infty}= k$, $\|\mathbf{j}\|_\infty=1$, $|\mathbf{m}|\neq |\mathbf{j}|$, and $\langle \mathbf{j}^\perp, \mathbf{m}\rangle \neq 0$. 
\end{itemize}
For those such $\mathbf{k}$ away from the axes and the lines $|y|=|x|$ in the $(x,y)$-plane, take $\mathbf{j} \in H_{N}$ to be the unique member of the set $\{ (1,0), (0,1), (-1,0), (0, -1)\}$ such that $\|\mathbf{k}-\mathbf{j}\|_{\infty}=k$.  Thus define $\mathbf{m}=\mathbf{k}-\mathbf{j}$ and note that $\mathbf{j}$ and $\mathbf{m}$ have different Euclidean lengths and $\langle \mathbf{j}^{\perp}, \mathbf{m}\rangle \neq 0.$ Now suppose $\mathbf{k}$ is on one of the axes or the lines $|y|=|x|$.  Then there exists $\mathbf{j} \in \{(1,0), (0,1), (-1,0), (0,-1) \}$ such that $\mathbf{m}=\mathbf{k}-\mathbf{j}$ belongs to the set of indices generated up to this point of sup norm length $k+1$.  It is easy to check that, again, $\mathbf{j}$ and $\mathbf{m}$ have different Euclidean lengths and $\langle \mathbf{j}^{\perp},\mathbf{m} \rangle\neq 0.$ This finishes the proof of the inductive argument.

Now note that we may choose a basis of $\mcC$ such that 
\begin{align*}
\mcD(w, q) = \CC^{8N(N+1)}
\end{align*}
for all $(w, q) \in \CC^{8N(N+1)}$.  Moreover, the origin is clearly an equilibrium point of $\mcG$.  Because the issue of explosion is still evident, Theorem \ref{dsaimppos} implies that for every $(w,q), (w', q') \in \CC^{8N(N+1)}$ and $T>0$, there exists $N\in \NN$ large enough such that  
\begin{align*}
p_t^n((w,q), (w',q')) >0 \, \, \text{ for all } t \geq T, \, n\geq N.  
\end{align*}
for all $t>0$.  
\end{proof}
\end{example}

\section{Proof of Main Results}
\label{sec:pres}
The goal of this section is to prove Theorem \ref{dsaimppos} and Theorem \ref{thm:posinv}.  Theorem \ref{thm:posinv} will be a relatively straightforward consequence of Theorem \ref{dsaimppos}, so we focus our attention first on proving Theorem \ref{dsaimppos}.

To prove Theorem \ref{dsaimppos}, we will use a slight modification of the condition for positivity of the density given by Ben Arous and L\'{e}andre \cite{BAL} (see also \cite{Nua98}).  The slight modification is necessary to remove the global Lipschitzian and boundedness conditions often assumed of the coefficients in the SDE.  

To setup the statement of our slight modification, let $H_\cdot= \int_0^\cdot h_s \, ds$, $h \in L^2([0,\infty): \RR^r)$, and $\Phi_\cdot^x(H)$ denote the maximally-defined solution (in time) of the equation   
\begin{equation}\label{controlp}
\Phi_s^x(H) = x+ \int_0^s X_0(\Phi_u^x(H))\, du  + \sum_{j=1}^{r} X_{j}\int_0^s h_u^{j}\, du. 
\end{equation} 
$J_{s,t}^x(H)$ denotes the maximally-defined $d\times d$ matrix-valued solution of  
\begin{eqnarray}
J_{s,t}^x(H) = \text{Id}_{d\times d} + \int_{s}^{t} DX_0 (\Phi_u^x(H)) J_{s,u}^x(H) \, du  
\end{eqnarray}
where $\text{Id}_{d\times d}$ is the identity matrix and $D$ is the
Jacobian.  Define the Gramian matrix $M_t^x(H)$ by 
\begin{equation}
(M_t^x(H))^{nk} = \sum_{m=1}^{r} \int_{0}^{t} (J_{s,t}^x(H)X_m)^{n}(J_{s,t}^x(H) X_m)^{k} \, ds.
\end{equation}  
\begin{remark}  
Sometimes $M_t^x(H)$ is called the deterministic Malliavin  covariance
matrix. Formally replacing $H$ with a Brownian motion $W$ yields the
standard (stochastic)  Malliavin  covariance matrix.
\end{remark}
\begin{lemma}
\label{lem:smbal}
Fix $x, z\in \RR^d$ and $t>0$ and suppose that $H_\cdot = \int_0^\cdot h_s\, ds$, $h\in L^{2}([0, \infty): \RR^r)$, is such that $\Phi_s^x(H)$ is defined for all times $s\in [0,t]$ and $\Phi_t^x(H)=z$.  If $M_t^x(H)$ is invertible, then 
\begin{align*}
p_t^n(x,z) >0
\end{align*}
for any integer $n\geq 1$ such that $\Phi_s^x(H) \subset B_n(0)$ for all $s\in [0,t]$.       
\end{lemma}

We defer the proof of Lemma \ref{lem:smbal} until the Appendix, and
focus our efforts in this section on exhibiting a control $H_\cdot =
\int_0^\cdot h_s \, ds$, $h\in L^2([0, \infty): \RR^r)$, so that
$\Phi_\cdot^x(H)$ has all of the properties stated in Lemma
\ref{lem:smbal}.  The proof of the existence of such a control splits into two parts.  First, in Section \ref{sec:geomcontrol} we will use the enlargement techniques of Jurjevic and Kupka \cite{JK81, JK85,Jur97} to see which directions can be flowed along in small times by
$\Phi_s^x(H)$ over the class of controls $H$ defined above.  Second,
we will see that there are enough directions so that we can construct
a sufficiently ``twisty" control $H$, ensuring that $M_t^x(H)$ is
invertible.  The existence of an equilibrium point $y\in \RR^d$ as in
the statement of Theorem \ref{dsaimppos} allows us control over the
time parameter.

\subsection{A Primer on Geometric Control Theory}
\label{sec:geomcontrol}

For $x\in \RR^d$ and $t>0$, let $A(x, \leq t)$ be the set of points $z\in \RR^d$ such that for some time $t_0\in(0, t]$ there exists $H_\cdot =\int_0^\cdot h_s \,ds $, $h\in L^2([0, \infty): \RR^r)$, for which $\Phi_s^x(H)$ is defined for all $s\in [0, t_0]$ and $\Phi_{t_0}^x(H)=z$.  Recalling the set $\mcC$ defined in Section \ref{sec:mainresults}, here we will use the techniques \cite{JK81, JK85, Jur97} to prove the following result:
\begin{lemma}
\label{lem:dsagen}
For all $x\in \RR^d$ and all $t>0$, $\{ x\} + \mcC \subset \overline{A(x, \leq t)}$.  
\end{lemma}  
We start by making some heuristic observations, arguing intuitively why we should expect Lemma \ref{lem:dsagen} to be true.  To make notation more legible, for any $C^\infty$ vector field $V$ on $\RR^d$ let $\exp(tV)(x)$ denote the maximally-defined integral curve of $V$ passing through $x$ at $t=0$.

We first see why we should expect the following containment to hold 
\begin{align}
\label{eqn:cont1}
\{x\} +\text{span}\{X_1, \ldots, X_r\} \subset \overline{A(x, \leq t)}
\end{align}
for all $x\in \RR^d$, $t>0$.  Let $x\in \RR^d$, $\alpha \in \RR\setminus \{0\}$ and $j\in \{ 1, \ldots, r\}$ be given. The key is to realize that for $\lambda >0$ large and $t>0$ small 
\begin{align*}
\exp(t(X_0 + \alpha \lambda X_j))(x) \approx \exp(t \alpha \lambda X_j)(x)
\end{align*}
This is because the behavior of the flow along $X_0 + \alpha \lambda X_j$ is initially dominated for small times by the flow along $\alpha \lambda X_j$ since $\lambda$ is large. More precisely, taking $t= t'/\lambda$ for some $t'>0$ fixed, one can show that as $\lambda \rightarrow \infty$
\begin{align*}
\exp(t(X_0 + \alpha \lambda X_j))(x) = \exp\Big(\frac{t'}{\lambda}(X_0 + \alpha \lambda X_j)\Big)(x) \rightarrow \exp(t' \alpha X_j)(x).
\end{align*}
Since $x\in \RR^d$, $\alpha \in \RR\setminus \{0 \}$ and $j\in \{1,2, \ldots, r\}$ were assumed to be arbitrary, we now see why one should believe the containment \eqref{eqn:cont1} as one could repeat the same argument with $\alpha X_j$ replaced by an arbitrary linear combination of $X_1, \ldots, X_r$.

To see how some of the commutators in the definition of $\mcC$ arise, we start by ``tweaking" the directions $X_1, \ldots, X_r$ obtained in the previous step by $X_0$; that is, we will first flow along $X_j$ for $\alpha \lambda$ units of times and then flow along $X_0$ for $t>0$ units of time.   Again let $x\in \RR^d$, $\alpha \in \RR\setminus\{0\}$ and $j \in \{ 1, \ldots, r\}$ be given.  If $x_j\in \RR^d$ is the constant value of $X_j$, we notice that for $t>0$ small
\begin{align}
\label{eqn:flowout}
\exp( t X_0) \circ \exp( \alpha \lambda X_j)(x) &= \exp(t X_0)( x + \alpha \lambda x_j)\\
\nonumber &= x+ \alpha \lambda x_j + \int_0^t X_0 ( x + \alpha \lambda x_j + \mathcal{O}(s))\, ds.   
\end{align}      
Letting $t = t'/\lambda^{n(X_j, X_0)}$, it follows that  as $\lambda \rightarrow \infty$
\begin{align}
\label{eqn:bakercamphaus}
 \int_0^t X_0 (x + \alpha \lambda x_j + \mathcal{O}(s))\, ds\rightarrow \frac{\alpha^{n(X_j, X_0)}}{n(X_j, X_0)!} \text{ad}^{n(X_j, X_0)}X_j(X_0)(x).  
\end{align}
As much as we would like to obtain this potentially new direction by taking $\lambda \rightarrow \infty$ in \eqref{eqn:flowout}, we cannot as $\alpha \lambda x_j$ blows up as $\lambda \rightarrow \infty$.  To rid ourselves of this problem, we need to flow backwards along $X_j$ for $\alpha \lambda$ units of time producing the relation 
\begin{align*}
&\exp(-\alpha \lambda X_j) \circ\exp( t X_0) \circ \exp( \alpha \lambda X_j)(x)  \\
&= x+  \int_0^t X_0 ( x + \alpha \lambda x_j + \mathcal{O}(s))\, ds .
\end{align*}   
Using the same scaling of time $t=t'/\lambda^{n(X_j, X_0)}$,  we now see how the commutator on the righthand side of \eqref{eqn:bakercamphaus}, hence in the definition of $\mcG_{1}^{\text{e}}$ and $\mcG_{1}^{\text{o}}$, arises.  

\begin{remark}
\label{rem:split}
Note that this computation explains why the separation of $\mcC$ into $\mcC^\text{o}$ and $\mcC^\text{e}$ is needed.  If $n(X_j, X_0)$ is even and $\text{ad}^{n(X_j, X_0)}X_j(X_0)$ is constant, then relation \eqref{eqn:bakercamphaus} implies that we may only flow along $\text{ad}^{n(X_j, X_0)}X_j(X_0)$ for positive times.  Additionally, in the subsequent iteration of this method we cannot necessarily flow backwards along this vector field producing yet another direction.   
\end{remark}

\begin{remark}\label{nonExample}
Following these observations, it is evident where and why Theorem
\ref{dsaimppos} will fail to either produce optimal results or be
applicable at all.  The failure is precisely due to the fact that the
set $\mcC$ only includes those constant vector fields which can be
flowed along in small positive times.  In particular, Theorem
\ref{rem:notoptimal} does not account for cases where there is an
unavoidable time delay needed to access certain points in space (as in
the example highlighted in Remark \ref{rem:notoptimal}), usually due
the need
to employ the drift vector field $X_0$.  Moreover, Theorem
\ref{rem:notoptimal} will not even apply in situations if there is a
more serious absence of time reversibility preventing $\mcC$ from
being $d$-dimensional.  As an example, consider the following SDE on $\RR^3$
\begin{equation}
\begin{aligned}
dx_t& = -x_t y_t \,dt + dB_t \\
dy_t & = (x_t^2 -y_t z_t) \, dt \\
d z_t &=(y_t^2 - z_t )\, dt.  
\end{aligned}
\end{equation} 
For this system, it is not hard to check that H\"{o}rmander's bracket condition is satisfied globally but 
\begin{align*}
\mcC = \{ \alpha \partial_x + \lambda \partial_y \, : \, \alpha \in \RR, \, \lambda \geq 0 \}.  
\end{align*} 
Hence, Theorem \ref{dsaimppos} does not apply since $\mcC$ has dimension $2 <3$.  

Even though our general result does not apply in this example, computing $\mcC$ is still useful in that Lemma \ref{lem:dsagen} is true regardless if $\mcC$ is $d$-dimensional.  If $\mcC$ is not $d$-dimensional, one can now proceed to find more points in the set $\overline{A(x, \leq t)}$ by using $\mcC$ and the specific nature of the drift vector field $X_0$.  Then, given the existence of $H_\cdot = \int_0^\cdot h_s \,ds$, $h\in L^2([0, \infty): \RR^r)$ such that $\Phi_t^x(H)=z$, positivity of the transition density $p_t^n(x,z)$ for $n$ large enough can then be shown by following a similar line of reasoning to Lemma \ref{lem:Minv} or Remark \ref{rem:brownianpath}.    
\end{remark}

We now turn the previous heuristics into a proof of Theorem \ref{lem:dsagen}.  Our proof will employ results from the reference \cite{Jur97}, so we will first introduce some further notation and terminology to connect with the setup there.  

We recall that for any $C^\infty$ vector field $V$ on $\RR^d$,
$\exp(tV)(x)$ denotes the maximally defined integral curve of $V$
passing through $x$ at time $t=0$.  Let $\mcH$ be any set of $C^\infty$
vector fields on $\RR^d$.  For $x\in \RR^d$ and $t>0$, $A_\mcH(x, \leq
t)$ denotes the set of $z \in \RR^d$ such that there exist positive
times $t_1, \ldots, t_k$ and corresponding vector fields $V_1, \ldots,
V_k \in \mcH$ such that $t_1 + \cdots +t_k \leq t$ and
\begin{align*}
\exp(t_k V_k) \circ \exp(t_{k-1} V_{k-1}) \circ \cdots \exp(t_1 V_1)(x)=z.  
\end{align*}
Because there will be many different sets of vector fields, here we will absolutely need to emphasize the dependence of these sets on $\mcH$.  

Two sets of $C^\infty$ $\RR^d$-vector fields, $\mcH$ and $\mcI$, are called
\emph{equivalent}, denoted by $\mcH \sim \mcI$, if
$\overline{A_{\mcH}(x,
  \leq t)}= \overline{A_{\mcI}(x, \leq t)}$ for all $x\in
\RR^d$ and all $t>0$.  One can show, see \cite{Jur97}, that
if $\mcH \sim \mcI$ and $\mcH\sim \mcJ$,
then $\mcH\sim  \mcI\cup \mcJ$.  In particular, if
we define
\begin{equation*} \sat(\mcH)= \bigcup_{\mcI \sim
\mcH} \mcI,
\end{equation*} then it also follows that $\sat(\mcH)\sim \mcH$.
$\sat(\mcH)$ is called the \emph{saturate} of $\mcH$. 

\begin{remark}
It is often the case that $\sat(\mcH)$ contains more vector fields than $\mcH$ itself.  Moreover, the saturate maintains identical accessibility properties in the sense $(\sim)$ described above.  This is convenient in that it allows one to use simpler vector fields to determine accessibility properties of the original set of vector fields $\mcH$.   For example, even though the constant vector field $X_j$, $j\geq 1$, does not belong to  
\begin{equation*}
  \mcG = \{ X_0 + \textstyle{\sum}_{j=1}^r u_j X_j\, : \,
u_j \in \RR\},
\end{equation*}
we used it above to generate more directions in $\overline{A(x, \leq
  t)}$ as done in the arguments following equation
\eqref{eqn:flowout}.  Using a limiting procedure, however, one can
justify that this is indeed permissible.
\end{remark}

In the next two lemmas, we list operations which allow us to expand (up to equivalence) a set of vector fields $\mcH$.  

\begin{lemma}\label{conelem} $\mcH$ is equivalent to the closed
convex hull of the set
\begin{equation*} \{\lambda V\, : \, \lambda \in [0,1],\, V \in
\mcH \}.
\end{equation*} Here the closure is taken in the topology of uniform
convergence with all derivatives on compact subsets of
$\RR^d$.
\end{lemma}

\begin{proof} Apply Theorem 5 and Theorem 6 in Chapter 2 of
\cite{Jur97}.
\end{proof}

 To state the next lemma, let $\psi: \RR^d\rightarrow \RR^d$ be a diffeomorphism.  For any $V\in \mcH$, we may define a vector field $\psi_{*}(V)$ by 
\begin{equation*}
 \psi_{*}(V)(x)= D\psi (\psi^{-1}(x)) V(\psi^{-1}(x))
\end{equation*}
where $D\psi$ is the Jacobian of $\psi$.  A diffeomorphism $\psi: \RR^d\rightarrow \RR^d$ is called a \emph{normalizer} of $\mcH$ if $\psi(x), \psi^{-1}(x) \in \overline{A_{\mcH}(x, \leq t)}$ for all $x\in \RR^d$ and all $t>0$.  The set of normalizers of $\mcH$ is denoted by $\text{Norm}(\mcH)$.

\begin{lemma}\label{normlem}
\begin{equation*} \mcH \sim \bigcup_{\psi\in \text{Norm}(\mcH)}\{\psi_{*}(V) \, : \, V\in
\mcH \}.
\end{equation*}
\end{lemma}

\begin{proof}
  Notice that by the lemma immediately after Definition 5 of Chapter 2
  of \cite{Jur97}, if $\psi$ is a normalizer of $\mcH$ using our
  definition, then it is also a normalizer using the definition given
  in \cite{Jur97}.  The result then follows after applying Theorem 9
  in Chapter 2 of \cite{Jur97} and using the fact that the identity
  map is a normalizer.
\end{proof}

\begin{remark}
  We will see in the proof of Lemma \ref{lem:dsagen} that the limiting
  procedure used in our heuristic calculations is exactly of the type
  covered by Lemma \ref{conelem}.  We will also see that the use of
  normalizers is very much in line with one's ability to flow along a constant vector field for positive or negative times (hence the
  $\psi$ and $\psi^{-1}$ in the definition of a normalizer).
\end{remark}

Using repeated applications of Lemma \ref{conelem} and Lemma
\ref{normlem}, we now prove Lemma \ref{lem:dsagen}.

\begin{proof}[Proof of Lemma \ref{lem:dsagen}]
  Let $\mcG= \{ X_0 + \textstyle{\sum}_{j=1}^r u_j X_j \, : \, u_j \in
  \RR\}$.  First note that it suffices to show that if $V\in
  \mcC^{\text{o}}$ and $W\in \mcC^{\text{e}}$, then $\alpha V,
  \lambda W \in \sat(\mcG)$ for all $\alpha \in \RR$ and all $\lambda
  \geq 0$.  The result would then follow by Lemma \ref{conelem} since
  if $V_1, V_2, \ldots, V_k \in \mcC^{\text{o}}$ and $W_1, W_2,
  \ldots, W_j\in \mcC^{\text{e}}$, then
\begin{align*}
  \sum_{l=1}^k \alpha_l V_l +\sum_{i=1}^{j} \lambda_i W_i\in \sat(\mcG)
\end{align*}     
for all $\alpha_i \in \RR$ and all $\lambda_i \geq 0$.  

We first demonstrate that $\alpha X_j \in \sat(\mcG)$ for all $\alpha
\in \RR$ and $j \in \{1, \ldots, r\}$.  Indeed, by Lemma \ref{conelem}
we have
\begin{equation*} \alpha X_{j} = \lim_{\lambda \rightarrow \infty}
\frac{1}{\lambda}(X_{0} + \alpha \lambda X_{j})\in \sat(\mcG).
\end{equation*}
By induction, it is enough to show that if $V$ is a constant vector
field with $\alpha V \in \sat(\mcG)$ for all $\alpha\in \RR$ and $W\in
\sat(\mcG)$ is a polynomial vector field, then
\begin{align*}
\frac{\alpha^{n(V, W)}}{n(V, W)!} \text{ad}^{n(V, W)}V(W) \in \sat(\mcG)
\end{align*}
for all $\alpha \in \RR$.  To prove this result, we seek to apply
Lemma \ref{normlem}.  Since $V$ is a constant vector field, let
$v=V(x)\in \RR^d$ denote its constant value.  For $\alpha \in \RR$,
define a map $\psi_\alpha : \RR^d\rightarrow \RR^d$ by
 \begin{align*}
 \psi_\alpha(x) = x - \alpha v.  
 \end{align*}      
 Note that, for each $\alpha \in \RR$, $\psi_\alpha$ is a normalizer
 for $\mcG$.  Hence, for each $\alpha \in \RR$, Lemma \ref{normlem}
 implies that $(\psi_{ \alpha})_{*}(W) \in \sat(\mcG)$.  Since
 $D\psi_\alpha$ is the identity matrix, notice that
\begin{equation*} (\psi_{\alpha})_{*}(W)(x)=W(x+ \alpha
v). 
\end{equation*} 
Applying Lemma \ref{conelem}, we thus find that for all $\alpha \in \RR$
\begin{align*}
V_{\alpha W}:=\lim_{\lambda\downarrow 0}\frac{1}{\lambda^{n(V, W)}} (\psi_{\lambda \alpha})_{*}(W) \in \sat(\mcG).  
\end{align*}    
To finish the proof, all we must see is that 
\begin{align*}
V_{\alpha W}= \frac{\alpha^{n(V,W)}}{n(V, W)!} \text{ad}^{n(V, W)}V(W).    
\end{align*}
Recalling that $v \in \RR^d$ denotes the constant value of $V$, for $x \in \RR^d$ fixed consider the function $F : \RR \rightarrow \RR^d$ defined
by $\alpha\mapsto W (x+ \alpha v)$.  By induction, for $j\geq 1$
\begin{equation*} F^{(j)}(\alpha)=
\text{ad}^{j}V(W)(x + \alpha v).
\end{equation*} where $F^{(j)}$ is the $j$th
derivative of $F$ with respect to $\alpha$.  Hence we
obtain the formula
\begin{eqnarray*} (\psi_{\alpha})_{*}W(x)=
F(\alpha) = \sum_{j=0}^{n(V,W)}
\frac{\alpha^{j}}{j!}F^{(j)}(0) = \sum_{j=0}^{n(V,W)}
\frac{\alpha^{j}}{j!}\text{ad}^{j}V(W)(x)
\end{eqnarray*} since each component of $F(\alpha)$ is a polynomial in $\alpha$ with degree $ \leq n(V,W)$.  Hence we now see that
\begin{equation*} 
V_{\alpha W}=\lim_{\lambda \rightarrow \infty}\frac{1}{\lambda^{n(V,W)}}(\psi_{\alpha \lambda})_{*}(W)
=\frac{\alpha^{n(V, W)}}{n(V, W)!}\text{ad}^{n(V,W)}  V(W),
\end{equation*}
completing the proof.  
\end{proof}

Before proceeding onto the second part of the argument, we state the following lemma which we will need later.


\begin{lemma}\label{intlielem}
  Suppose that, for some $x\in \RR^d$, the Lie algebra generated by $\mcH$ evaluated at $x$ spans the tangent space.  Then for all $t, \epsilon>0$
\begin{equation*}
  \text{\emph{interior}}(A_{\mcH}(x, \leq t+ \epsilon) ) \supset
  \text{\emph{interior}}(\overline{A_{\mcH}(x, \leq t)}).
\end{equation*}
\end{lemma}
\begin{proof} See Theorem 2 of Chapter 3 in \cite{Jur97}.
\end{proof} 

\subsection{Strict Positivity}

The next two lemmas will operate as an easy-to-check criterion
assuring that, for a given control $H$, $M_t^x(H)$ is invertible.
Though not necessary (see Remark \ref{rem:brownianpath}), these
results use the fact that $\mcG$ contains only polynomial vector
fields.  In particular, the special structure of zero sets of
polynomials is employed in the following lemma.

\begin{lemma}\label{lem:openset}  Suppose that $\mcC$ is $d$-dimensional and let $\mcH = \cup_{m =1}^r \{ X_m, [X_0, X_m]\}.$  Then for any non-empty open $A \subset \RR^d$ the set of points in $\RR^d$ given by 
\begin{equation}\label{foliation} \bigcup_{x\in A} \{V(x) \, : \, V \in \mcH\}
\end{equation}
is $d$-dimensional.  
\end{lemma}

\begin{proof} Suppose that the subspace spanned by the set in \eqref{foliation} has dimension $l\leq d$ and choose a basis $v_{1}, v_{2}, \ldots, v_{l}\in \RR^d$ for this subspace.  The goal is to show that $l=d$.  Let $V_{1}, V_{2}, \ldots, V_{l}$
be the constant vector fields with constant values $v_{1}, v_{2},
\ldots, v_{l}$, respectively.  Notice that every vector field $V$ in the span of $\mcH$ is a polynomial vector field and satisfies the following equality on the open set $A$
\begin{equation}\label{lieb} V= p_1 V_1 + p_2 V_2 + \cdots + p_l V_l
\end{equation} 
for some polynomials $p_1, \ldots, p_l$.  Since $A$ is open and $V$ is a polynomial vector field, \eqref{lieb} is valid everywhere on $\RR^d$.  
Moreover, since vector fields of the form \eqref{lieb} are closed under commutators and linear combinations, we see that 
\begin{equation*} \text{span} (\mcC) \subset \text{span}\{v_1, v_2,
\ldots, v_{l} \}
\end{equation*} 
Note that this finishes the proof since $\mcC$ is $d$-dimensional.  
\end{proof}

To setup the statement of the next result, define $K^x_t(H)\subset \RR^d$ as follows:
\begin{equation}
  K^x_t(H) = \bigcup_{m=1}^{r}\big\{X_m(\Phi_s^x(H)), [X_0, X_m](\Phi_s^x(H))\, :
  \, s \in (0,t)\big\}.  
\end{equation}    
\begin{lemma}
\label{lem:Minv}
Suppose that $K^x_t(H)$ is $d$-dimensional. Then the associated matrix
$M_t^x(H)$ is invertible.
\end{lemma}

\begin{proof}
  It suffices to show that $M_{t}^x(H)$ is positive definite.  Assume,
  to the contrary, that $M_t^x(H)$ is not positive-definite and let
  $\langle \, \cdot\,, \, \cdot\, \rangle$ denote the inner product on
  $\RR^d$. Then there exists $y\in \RR^d \setminus \{ 0\}$ such that
\begin{eqnarray*} 
0=\langle M_{t}^x(H) y, y \rangle = \sum_{m=1}^{r}
\int_{0}^{t} \langle J_{s,t}^x(H) X_{m}, y \rangle^{2}\, ds.  
\end{eqnarray*}
To get a contradiction, we seek to obtain a lower bound $\langle
M_{t}^x(H) y, y \rangle $ which is positive using the equality above.
To derive such a bound, first observe that for $s\leq s_0\leq u_0 \leq
t_0 \leq t$, $J_{s_{0},t_0}^x(H)= J_{u_0, t_0}^x(H) J_{s_0, u_0}^x(H)$
and that the matrix $J_{s_0,t_0}^x(H)$ is invertible.  Using these two
facts, it is not hard to check that for $s\leq s_0 \leq t_0\leq t$
\begin{eqnarray}
\label{bwflow}\partial_{s_0} J_{s_0, t_0}^x(H) &=&- J_{s_0,t_0}^x(H) DX_0(\Phi_{s_0}^x(H))\\
\nonumber J_{t_0,t_0}^x(H)&=&\text{Id}_{d\times d}.     
\end{eqnarray}   
Letting $| \cdot |$ denote the Euclidean norm on
$\RR^d$, we then see that for all $u \in (0,t)$, $\epsilon \in (0, \min(u, t-u))$
\begin{eqnarray} 
\label{eqn:Mbound} 
\nonumber  0= \langle M_t^x(H) y, y \rangle & \geq&  \int_{0}^{t}  \ip{J_{s,t}^x(H) X_{m} , y}^{2}\, ds \\
\nonumber &\geq& \int_{u
  - \epsilon}^{u+ \epsilon}\ip{J_{s,t}^x(H) X_{m}, y }^{2}\, ds \\ 
&=&
\int_{u - 
  \epsilon}^{u+\epsilon}\ip{J_{s,u}^x(H) X_{m}, (J_{u, t}^x(H))^{*} y}^{2}\, ds
\\ \nonumber &\geq & | (J_{u,t}^x(H))^{*} y |^2 \inf_{y\, : \, \|y\| =1}
\int_{u-\epsilon}^{u+\epsilon} \ip{J_{s,u} X_m, y}^2 \, ds.  
\end{eqnarray}
Since $| J_{u,t}^* y| >0$ and the unit disk is compact in $\RR^d$, it
suffices to show that for all nonzero $y \in \RR^d$ there exists $m\in
\{1,2, \ldots, r\}$, $u \in (0, t)$, and $\epsilon \in (0, \min(u,
t-u))$ such that
\begin{equation} 
\label{eqn:festM}
\int_{u - \epsilon}^{u+ \epsilon} \ip{J_{s, u}^x(H)
X_{m}, y }^{2} \, ds>0.
\end{equation} 
Thus let $y\in \RR^d$, $y\neq 0$, be arbitrary.  By hypothesis, either
$\langle X_m, y\rangle \neq 0$ for some $m\in \{1, \ldots, r\}$ or
$\langle [X_m, X_0](\Phi_{t_0}^x(H)), y\rangle \neq 0$ for some $m \in
\{1, \ldots, r \}$, $t_0 \in (0,t)$.  Clearly, if $\langle X_{m},
y\rangle\neq 0$ for some $m\in \{1,2, \ldots, r\}$, then there is
nothing to show by continuity and \eqref{eqn:festM}.  Thus suppose
that $\langle X_{m}, y\rangle=0$ for all $m=1,2, \ldots, r$ and pick
$t_{0} \in (0, t)$, $m\in \{ 1,2, \ldots, r\}$ such that
\begin{equation*} \langle z, y \rangle =\langle [X_{m},
X_0](\Phi_{t_0}^x(H)), y \rangle \neq 0.
\end{equation*} Since $\langle X_{m}, y \rangle =0$, using the definition of $J_{s,t_0}^x(H)$ twice we see that
\begin{eqnarray*} \langle J_{s,t_0}^x(H)X_{m}, y \rangle &=&
\int_{s}^{t_0}\langle DX_{0}( \Phi_u^x(H) ) J_{s,u}^x(H) X_{m}, y\rangle \, du\\
&=& \int_{s}^{t_0} \langle[X_{m}, X_0](\Phi_u^x(H)), y \rangle \, du\\
&\, &  +
\int_{s}^{t_0}\bigg\langle DX_{0} ( \Phi_{u}^x(H)) \int_{s}^{u} DX_{0} (
\Phi_v^x(H) ) J_{s, v}^x(H)X_{m} \, dv, y\bigg\rangle\, du.  
\end{eqnarray*}  
Therefore, for $s$ sufficiently close to $t_0$, $\langle
J_{s,t_0}^x(H) X_{m}, y \rangle\neq 0$. Hence continuity then implies for any $\epsilon\in (0, t_{0})$
\begin{equation*} \int_{t_0 - \epsilon}^{t_0} \langle J_{s, t_0}^x(H)
X_{m}, y \rangle^{2}\, ds>0,
\end{equation*} 
finishing the proof.
\end{proof}

We now use the previous two results and Lemma \ref{lem:dsagen} to prove Theorem \ref{dsaimppos}.

\begin{proof}[Proof of Theorem \ref{dsaimppos}]
 We first prove Theorem \ref{dsaimppos} part (\ref{dsaimppos3}) and then show how part (\ref{dsaimppos1}) follows by a similar argument.  Therefore suppose that $y\in \RR^d$ is an equilibrium point of $\mcG$ and that $x,z\in \RR^d$ are such that $y\in \mcD(x)$ and $z\in \mcD(y)$.  By Lemma \ref{lem:smbal}, our goal is to exhibit $H_\cdot = \int_0^\cdot h_s \,ds$, $h\in L^2([0,t]: \RR^r)$, such that $\Phi_t^x(H)=z$ and $M_t^x(H)$ invertible.  To ensure that $M_t^x(H)$ is invertible, we will build $H_\cdot$ in such a way so as to ``twist" the path of $\Phi^x_\cdot (H)$ from $x$ to $z$.

We first claim that there exist countably many non-empty disjoint open subsets $U_l$, $l\geq 0$, with the property that 
\begin{align}
\label{eqn:contcon}
U_{l+1}\subset  \bigcup_{w\in U_l}\mcD(w)
\end{align} 
for all $l\geq 0$.  Suppose first that $\mcD(x)= \RR^d$.  Then it follows that $\mcD(x')=\RR^d$ for all $x' \in \RR^d$.  Thus in this case simply let $U_l$ be any partition of $\RR^d$.  If $\mcD(x) \neq \RR^d$, then since $y\in \mcD(x)$ write
\begin{align*}
y=x + \textstyle{\sum}_{j=1}^k \alpha_j y_j + \textstyle{\sum}_{j=k+1}^{d} \lambda_j y_j 
\end{align*}
for some $\alpha_j \in \RR$ and $\lambda_j>0$.  Let $\lambda = \min_{j} \lambda_{j}>0$ and define constants $\alpha_0=0$ and $\alpha_l= \sum_{k=1}^l 2^{-k}$, $l\geq 1$.  Note that for $l\geq 0$ the sets 
\begin{eqnarray*}
U_{l}
&= &x+ \text{span}\{y_1,\ldots, y_{k} \}+  \{\mu_{k+1} y_{k+1}  +  \cdots + \mu_{d} y_{d}  \, : \, \mu_{j}\in ( \alpha_l \lambda, \alpha_{l+1} \lambda) \}  
\end{eqnarray*}
are disjoint, open and satisfy \eqref{eqn:contcon}.  This finishes the proof of the claim.

By construction of the sets $U_l$, $l \geq 0$, and Lemma \ref{lem:openset}, there exist  $x_{l+r} \in U_{l}$ such that 
 \begin{equation*}
\bigcup_{m=1}^{r} \{x_1, \ldots, x_r, [X_{m}, X_0](x_{r+1}), \ldots, [X_m, X_0](x_{r+j})\}
 \end{equation*}
is $d$-dimensional.  Here, recall that $x_1, \ldots, x_r$ are the constant values of $X_1, \ldots, X_r$, respectively.  Moreover, $x_{r+1} \in \mcD(x)$, $y\in \mcD(x_{j+r})$ and
 \begin{equation*}
 x_{l+1+r} \in \mathcal{D}(x_{l+r})   
 \end{equation*} 
for all $l=1,2, \ldots, j$.

We now  show that we can build $H_\cdot$ so that the path $\Phi^x_\cdot(H)$ passes through each of these points prior to time $t>0$ and so that $\Phi_t^x(H)=z$.  Observe that Lemma
\ref{intlielem} and Lemma \ref{lem:dsagen} together imply $A(w,  \leq
s)\supset \mcD(w)$ for all $w\in \RR^d$ and all $s>0$.  Hence by
definition of $A(w, \leq s)$, there exist positive times $t_1, t_2,
\ldots, t_{j+1}$ with $\sum_{l=1}^{j+1} t_l < \frac{t}{2}$ and
corresponding $H_{l}(\cdot)= \int_0^\cdot h_{l}(s) \,ds$, $h_l \in
L^2([0, t_l]: \RR^r)$, such that $\Phi_{t_1}^x(H_1)=x_{r+1}$,
$\Phi_{t_{l+1}}^{x_{r+l}}(H_{l+1})= x_{r+l+1}$, $l=1, \ldots, j-1$,
and $\Phi_{t_{j+1}}^{x_{r+j}}(H_{j+1})=y$.  By piecing together the
$H_l$'s, this now gives us the path from $x$ to $y$.  For the rest of
the path, we may also pick a positive time $t_{j+3}< \frac{t}{2}$ and
$H_{j+3}(\cdot) = \int_0^\cdot h_{j+3}(s) \,ds$, $h_{j+3} \in
L^{2}([0, t_{j+3}]: \RR^r)$ such that $\Phi_{t_{j+3}}^{y}(H_{j+3})=z$.
Moreover, since $y$ is an equilibrium point of $\mcG$,
 letting
$t_{j+2}= t - (t_1 + \cdots + t_{j+1}+t_{j+3})>0$ there exists a
control $H_{j+2}(\cdot) = \int_0^\cdot h_{j+2}(s)\,ds$, $h_{j+2} \in
L^2 ([0, t_{j+2}]: \RR^r)$ such that $\Phi_{t_{j+2}}^y(H_{j+2})=y$.
By Lemma \ref{lem:Minv}, we now obtain the conclusion in part
(\ref{dsaimppos3}).

To prove part (\ref{dsaimppos1}), simply let $z=y$ in the first
argument and, for an arbitrary $T>0$, choose $t<T$.  Note that this now finishes the proof of Theorem \ref{dsaimppos}.
\end{proof}

\begin{remark}
  \label{rem:brownianpath}
  Without using the special structure of polynomial vector fields, one
  can prove Theorem \ref{dsaimppos} alternatively by choosing the path
  from $x$ to $y$ differently as follows.  Define 
 \begin{align*}
 D(x,y) = \begin{cases}
 \mcD(x) \setminus \mcD(y) & \text{ if } \mcD(x) \neq \RR^d\\
 \RR^d & \text{ otherwise}
 \end{cases}
 \end{align*} 
  and let $y' \in D(x,y)$ be arbitrary.  Since $D(x,y)$ is open, let $\delta >0$ be such that $B_\delta(y') \subset
  D(x,y)$.  By the support theorems \cite{SV1, SV},
  there exists $s_1 \in (0, t/4)$ such that for all $n$ large enough
\begin{align*}
\PP_x \{ s_1< \tau_n, \, x_{s_1} \in B_\delta(y')\}>0.    
\end{align*}       
Now recall that $W_s=(W^1_s,
\ldots, W^r_s)$ is an $r$-dimensional standard Wiener process defined
on the probability space $(\Omega, \mathscr{F}, \PP)$.  In this remark, we identify the set $\Omega$ with the space of continuous paths $C([0, \infty)
:\RR^r)$.  Letting $M_t^x(W(\omega))$ denote the matrix $M_t^x(H)$
when $H_s=(W^1_s(\omega), \ldots, W^r_s(\omega))$, we note that by
Malliavin's proof of H\"{o}rmander's theorem \cite{KS2, Nor86}
  \begin{align*}
    \PP_x\{s_1< \tau_n , \, x_{s_1} \in B_\delta(y'), \,
    M_{s_1}^x(W) \text{ invertible} \}=\PP_x \{ s_1<\tau_n, \, x_{s_1} \in B_\delta(y')\}>0
\end{align*}      
for all $n$ sufficiently large.  Therefore, fix 
\begin{equation*}
  \omega \in \{ s_1< \tau_n, \,  x_{s_1} \in B_\delta(y'), \,
  M_{s_1}(W(\omega)) \text{ invertible}\} 
\end{equation*}
and define $H_s= (W^1_s(\omega), \ldots, W^r_s(\omega))$ on the time
interval $[0, s_1]$.  Hence $\Phi_{s_1}^x(H)\in B_\delta(y')$.
Since $$ y\in \bigcap_{w\in B_{\delta}(y')}\mcD(w), $$ pick $\tilde{H}$ such
that for some $s_2 < \frac{t}{4}$
\begin{equation*}
  \Phi_{s_2}^{\Phi_{s_1}^x(H)}(\tilde{H})
  =y.  
\end{equation*}
We can complete our path from $y$ to $z$ in
exactly the same way as in the proof of Theorem \ref{dsaimppos}.
Invertibility of the covariance matrix for our chosen control at time
$t$ follows immediately since $M_{s_1}^x(W(\omega))$ is invertible. See Theorem 8.1 in \cite{MattinglyPardoux} for a similar argument.
\end{remark}

\begin{remark}
Yet another way to prove Theorem \ref{dsaimppos} is to use a Feynman-Kac representation of the probability density function $p_t^n(x,z)$.  Indeed fixing $n\in \NN$ and $x\in B_n(0)$, observe that the time-reversed density $q^{n}_s(x,z)=p_{t-s}^n(x,z)$ solves the following PDE
\begin{align*}
\frac{\partial q^{n}_s}{\partial s}  = - \mathcal{L}^{*}_z q_s^{n} \, \,\, \text{ on } \, \, \, [0, t) \times B_n(0)  
\end{align*}
where $\mathcal{L}^*_z$ is the formal adjoint (in the $z$ variable) of the Markov generator $\mathcal{L}$ corresponding to the diffusion $x_t$.  Now consider the process $y_t$ solving 
\begin{align*}
dy_t = - X_0(y_t) \, dt - \sum_{j=1}^r X_j \, dW_t^j 
\end{align*} 
and let $T_n = \inf\{t>0 \, : \, |y_t| \geq n \}$.  It then follows that we may write $p^n_t(x,z)$ as 
\begin{align*}
p^{n}_t(x,z) = q^n_0(x,z) = \E_z e^{\int_0^{s\wedge T_n} f(y_u) \, du} q_{s\wedge T_n} (x, y_{s\wedge T_n})
\end{align*}
for some $f\in C^\infty(\RR^d:\RR)$.  One can use now the expression above coupled with the support theorems \cite{SV1, SV} applied to the time-reversed process $y_t$ to bound $p^n_t(x,z)$ from below by a positive quantity.  
\end{remark}

We finish this section by proving Theorem \ref{thm:posinv} as a
consequence of Theorem \ref{dsaimppos} (\ref{dsaimppos1}).

\begin{proof}[Proof of Theorem \ref{thm:posinv}]
  Let $\mu$ be an invariant probability measure for the Markov process
  $x_t$ defined by \eqref{eqn:sde1}.  Again, since $\mcC$ is contained
  in the Lie algebra generated by $X_1, \ldots, X_r, [X_1, X_0],
  \ldots, [X_r, X_0]$ and $\mcC$ is $d$-dimensional, it follows by
  H\"{o}rmander's theorem \cite{Hor67} that $\mu(dx)=m(x) \, dx$ for
  some nonnegative function $m\in C^\infty(\RR^d)$.  Recall also that,
  for the same reasons, the Markov process $x_t$ defined by
  \eqref{eqn:sde1} has a probability density function $p_t(x,y)$ with
  respect to Lebesgue measure on $\RR^d$ which is smooth for $(t,x,y)
  \in (0, \infty) \times \RR^d \times \RR^d$.  Since $\mu$ is an
  invariant probability measure, we have the following relation for
  almost every $z\in \RR^d$ and $t>0$
\begin{align*}
m(z) = \int_{\RR^d} m(y) p_t (y,z) \, dy.  
\end{align*}
We now use this relation to prove the positivity assertion.  Let $x\in \supp(\mu)$.  Hence $\mu(B_\delta(x)) >0$ for all $\delta >0$.  By smoothness of the density $m$, for each $\delta>0$ there exists $x_1=x_1(\delta) \in B_{\delta}(x)$ such that $m(x_1) >0$.  Since $m$ is smooth, in particular continuous, there exists $\gamma >0$ such that $B_{\gamma}(x_1) \subset B_{\delta}(x)$ and $m(y) \geq \epsilon >0$ for all $y \in  B_{\gamma}(x_1) $.  Hence for almost every $z\in \RR^d$ we have
\begin{align*}
  m(z) \geq \int_{B_{\gamma}(x_1)} m(y) p_{t}(y,z) \, dy \geq \epsilon
  \int_{B_{\gamma}(x_1)} p_{t}(y,z) \, dy.
\end{align*} 
To bound $p_t(y,z)$ from below, there are two cases.  First suppose
that $\mcD(x)=\RR^d$.  Then by definition of $\mcD(x)$, we have that
$\mcC^{\text{o}}$ is $d$-dimensional, and hence $\mcD(y)=\RR^d$ for
all $y\in \RR^d$.  Theorem~\ref{dsaimppos}~(\ref{dsaimppos1}) implies
that for any $y\in B_\gamma(x_1), \, z\in \mcD(x)$ there exists $t>0$
such that $p_t(x,z)>0$.  Since the transition density is a continuous
function in all of its arguments, there exists an open neighborood $U$
of $(t, x,z)$ in $ (0, \infty) \times B_\gamma(x_1) \times \RR^d$ such
that $p_s(x',z')\geq c>0$ for $(s,x',z') \in U$.  In particular, for
almost every $y$ in an open ball centered at $z$
\begin{align*}
m(y) \geq \epsilon c >0.
\end{align*}
Since $m$ is continuous it follows that $m(z)\geq \epsilon c>0$.  For
the second case, suppose that $\mcD(x) \neq \RR^d$.  In particular,
this implies that $\mcC^{\text{o}}$ has dimension $l < d$ and
$x\notin \mcD(x)$. Take $z\in \mcD(x)$ and decrease $\delta >0$ so
that for every $y\in B_\delta(x)$, $z\in \mcD(y)$.  Following now in
the same way as in the previous case we finish the proof of the
result.

\end{proof}

\section*{Appendix}
Here we prove Lemma \ref{lem:smbal}.  We recall that this result is
the slight modification of the criterion for positivity of the density
given by Ben-Arous L\'{e}andre \cite{BAL} which was applied without
proof in Section \ref{sec:pres}.  Such an extension is needed in this
paper since the drift vector field $X_0$ was not assumed to be
globally Lipschitzian and its derivatives were not assumed to be
globally bounded.

The proof of Lemma \ref{lem:smbal} is almost identical to (and in some
parts simpler than) the proof of Proposition 4.2.2 of \cite{Nua98}.
The basic difference needed to remove these assumptions on $X_0$ is
that we need to compare the stopped process $x_{t\wedge \tau_n}$ with
another process $x_t^{(n)}$ such that $x_t^{(n)}$ solves an SDE whose
coefficients satisfy the required Lipschitzian and boundedness
conditions and $$x_{t\wedge \tau_n}= x_{t\wedge \tau_n}^{(n)} \,\,\text{ for all }\,\,t\geq 0.$$  This localization procedure is relatively standard but we include the details for completeness.

To do such a comparison, for any integer $n\geq 1$ let $X_0^{(n)}$ be
a $C^\infty$ vector field on $\RR^d$ satisfying
\begin{align*}
X_0^{(n)}(x)= \begin{cases}
X_0(x) & \text{ for } |x| \leq n\\
0 & \text{ for } |x|\geq n+1 
\end{cases}.
\end{align*}
For $x\in \RR^d$, $n\in \NN$, $t>0$ and $H=(H^j)\in C([0,t]: \RR^r) $ let $\Phi_t^{x,n}(H)$ denote the solution of the equation
\begin{align*}
\Phi_t^{x,n}(H)= x + \int_0^t X_0^{(n)}(\Phi_s^{x,n}(H))\, ds + \sum_{j=1}^r X_j  H^j_t .  
\end{align*}
Let $J_{s,t}^{x,n}=J_{s,t}^{x,n}(H)$ denote the $d\times d$ matrix-valued solution of the equation
\begin{align*}
J_{s,t}^{x,n} = \text{Id}_{d\times d} + \int_s^t DX_0^{(n)}( \Phi_u^{x,n}(H)) J_{s,u}^{x,n} \, du  
\end{align*}
and $M_t^{x,n}(H)$ denote the matrix
\begin{align*}
(M_t^{x,n}(H))_{lm}= \sum_{j=1}^r \int_0^t (J_{s,t}^{x,n}(H) X_j)^l (J_{s,t}^{x,n}(H) X_j)^m \, ds.   
\end{align*}

\begin{proof}[Proof of Lemma \ref{lem:smbal}]
  As in \cite{Nua98}, our goal is to use Malliavin calculus to bound $p_t^n(x,z)$ from below by a quantity which is positive if the covariance matrix
  $M_t^{x,n}(H)$ is invertible.  For brevity of notation during this
  proof, we will write the functional $\Phi_t^{x,n}( \, \cdot \,)$
  simply as $\Phi(\cdot\,)$.  Let $H_\cdot= \int_0^\cdot h_u \, du, \,
  h \in L^2([0,\infty):\RR^r)$ be as in the statement of the lemma and
  let $k_l(s)$ denote the $l$th row of the matrix $k_{lj}(s)=(J_{s,t}^{x,n}(H) X_j)_l$.  For $y\in \RR^d$, let
\begin{align*}
(T_y W)(t)= W(t) + \sum_{l=1}^d y_l \int_0^t k_l(s)\, ds  \,\,\, \text{ and } \,\,\,
g(y,W)= \Phi(T_y W)- \Phi(W) 
\end{align*} 
where $W(t)=(W^1(t), \ldots, W^r(t))$ denotes the standard $r$-dimensional Wiener process on $(\Omega, \mathscr{F}, \PP)$.  
For $\beta>1$, define cutoff functions $\mathcal{K}_\beta, \alpha_\beta \in C(\RR: [0,1])$ by 
\begin{align*}
\mathcal{K}_\beta(x)= \begin{cases}
0 & \text{ if } |x| \geq \beta\\
1 & \text{ if } |x| \leq \beta -1
\end{cases}
\,\,\, \text{ and } \,\,\, \alpha_\beta(x)= \begin{cases}
0 & \text{ if } |x| \leq \frac{1}{\beta}\\
1 & \text{ if } |x| \geq \frac{2}{\beta}
\end{cases},
\end{align*}
and set 
\begin{align*}
\mathcal{H}_\beta = \mathcal{K}_\beta(\| g(\cdot, W)\|_{C^2(B_1(0): \RR^d)}) \alpha_\beta(|\det \partial_j g^i(0)|).
\end{align*} 
Under our assumptions, one can check that (see \cite{NualartD2006},
Example 1.2.1, Theorem 2.2.2 and surrounding text) $g(\cdot, W(\omega)) \in C^\infty(\RR^d)$ for a.s. $\omega \in \Omega$.

Now let $f:\RR^d \rightarrow [0, \infty)$ be bounded, measurable and $\rho: \RR^r\rightarrow (0, \infty)$ be a measurable function satisfying $\int_{\RR^r} \rho(y) \, dy =1$.  Observe that 
 \begin{align*}
\E_x f(x_{t\wedge \tau_n}) &= \int_{\RR^r} \E_x f(x_{t\wedge \tau_n}) \rho(y) \, dy \\
&=  \int_{\RR^r} \E f(\Phi(W)) \boldsymbol{1}_{\{ \|\Phi(W)\|_{t} \leq n\}} \rho(y) \, dy  
\end{align*}
where 
\begin{align*}
\{ \| \Phi(W)\|_{t} \leq n\}= \bigg\{\omega\in \Omega\,: \,  \sup_{s\in [0,t]}|\Phi_s^{x,n}(W(\omega))| \leq n\bigg\}.
\end{align*}
Girsanov's theorem then gives
\begin{align*}
&\int_{\RR^r} \E f(\Phi(W)) \boldsymbol{1}_{\{\|\Phi(W)\|_t \leq n \}} \rho(y) \, dy\\
& = \int_{\RR^r}\E f(\Phi(T_y W)) \boldsymbol{1}_{\{ \| \Phi(T_y W)\|_{ t }\leq n\}} G(y)\rho(y) \,dy   
\end{align*}
where $G(y)>0$ is the Radon-Nikodym derivative in the Girsanov change of measure formula.  Using this equality we see that for any $c_\beta
>0$ 
\begin{align*}
\E_x f(x_{t\wedge \tau_n}) &\geq \int_{\RR^r}\E f(\Phi(T_yW))
\boldsymbol{1}_{\{ \| \Phi(T_y W)\|_{ t }\leq n\}} G(y) \rho(y) \,dy
\\ 
&\geq \E \mathcal{H}_\beta \int_{|y| \leq c_\beta} f(g(y)+ \Phi(W))
\boldsymbol{1}_{\{ \| \Phi(T_yW)\|_{ t }\leq n\}}     G(y) \rho(y) \,
dy  \\ 
&\geq \EE  \mathcal{H}_\beta \textbf{1}_{\{\sup_{|y| \leq c_\beta}
  \|\Phi(T_y W)\|_t \leq n\}} \int_{|y| \leq c_\beta} f(g(y) +
\Phi(W)) G(y) \rho(y) \, dy. 
\end{align*}
Let $A_{\beta}= \{ \sup_{|y| \leq c_\beta} \| \Phi(T_yW)\|_t \leq n
\}$.  By Lemma 4.2.1 of \cite{Nua98}, for any $\beta >1$ there exist
constants $c_\beta\in (0, \beta^{-1})$ and $\delta_\beta >0$ such that
any mapping $G: B_1(0) \rightarrow \RR^d$ with $G(0)=0$,
$\|G\|_{C^2(B_1(0))}\leq \beta$ and $|\det \partial_j g^i(0)| \geq
\frac{1}{\beta}$ is diffeomorphic from $B_{c_\beta}(0)\subset \RR^d$
into a neighborhood of $B_{\delta_\beta}(0) \subset \RR^d$.  In
particular, we find that after changing variables twice
\begin{align*}
  \E_x f(x_{t\wedge \tau_n}) &\geq  \E \mathcal{H}_\beta  1_{A_\beta} \int_{|y| \leq c_\beta} f(g(y) + \Phi(W)) G(y) \rho(y) \, dy\\
  &\geq \E \mathcal{H}_\beta 1_{A_\beta} \int_{|z| \leq \delta_\beta}
  \hspace{-.04in}f(z + \Phi(W)) G(g^{-1}(z)) \rho(g^{-1}(z))
  |\det \partial_j
  g^i(g^{-1}(z))|\, dz\\
  &= \E \mathcal{H}_\beta 1_{A_\beta} \int_{|z-\Phi(W)| \leq
    \delta_\beta} f(z) G(g^{-1}(z-\Phi(W)))
  \\&\qquad\qquad\qquad\qquad\times \rho(g^{-1}(z-\Phi(W)))
  |\det \partial_j g^i(g^{-1}(z-\Phi(W)))| \, dz.
\end{align*}
Therefore we deduce the following inequality 
\begin{align*}
p_t(x,z) \geq& \E \mathcal{H}_\beta \boldsymbol{1}_{A_\beta}
\boldsymbol{1}_{\{ |z- \Phi(W)| \leq \delta_\beta \}}
G(g^{-1}(z-\Phi(W)) \\&\qquad\qquad\qquad\qquad\times \rho(g^{-1}(z-\Phi(W))) |\det \partial_j
g^i(g^{-1}(z- \Phi(W))|. 
\end{align*}
By construction, if $\mathcal{H}_\beta \neq 0$ and $|z- \Phi(W) | \leq \delta_\beta$ then
\begin{align*} 
  G(g^{-1}(z-\Phi(W)) \rho(g^{-1}(z-\Phi(W))) |\det \partial_j
  g^i(g^{-1}(z- \Phi(W))| >0.
\end{align*}
Thus it remains to prove that $\beta >0$ can be chosen large enough so that the event  
\begin{multline*}
A_\beta \cap \Big\{|z-\Phi(W)| \leq \delta_\beta, \, |\det \partial_j
g^i(0)| \geq 2\beta^{-1}, \,\|g(\cdot, W)\|_{C^2(B_1(0))}\leq
\beta-1\Big\}    
\end{multline*}
has positive probability.  Note that this can be shown by following
exactly the same line of reasoning starting in the last paragraph of
p. 1777 of \cite{MattinglyPardoux}.

\end{proof}

\bibliography{control}{} \bibliographystyle{plain}

\end{document}